\documentclass[a4paper,10pt,twoside]{article}
 
\usepackage{amsmath,amssymb,amsthm,dsfont}
\usepackage[english]{babel}
\usepackage[bookmarks]{hyperref} 
\usepackage{mathrsfs} 
\usepackage{color,soul} 

\usepackage{geometry}
\newtheorem{theorem}{Theorem}[section]

\newtheorem{lem}[theorem]{Lemma} 
\newtheorem{prop}[theorem]{Proposition}
\newtheorem{coro}[theorem]{Corollary} 
\theoremstyle{definition}
\newtheorem{rem}[theorem]{Remark}

\newcommand{\N}{\mathbb N}
\newcommand{\Z}{\mathbb Z}

\newcommand{\R}{\mathbb R}
\newcommand{\C}{\mathbb C}

\newcommand{\Sh}{\mathscr{S}}
\newcommand{\Ph}{\mathscr{P}}
\newcommand{\ep}{\epsilon}

\newcommand{\scal}[1]{\left\langle #1 \right\rangle} 
\newcommand{\name}{$\underline{\qquad \qquad}$} 

\newcommand{\defendproof}{\hfill $\Box$} 

\usepackage{fancyhdr} 
\begin{document}
\title{\sc Well-posedness of nonlinear
fractional Schr\"odinger and wave equations in Sobolev spaces}
\author{\sc{Van Duong Dinh}} 
\date{ }
\maketitle

\begin{abstract}
We prove the well-posed results in sub-critical and critical cases for the pure power-type nonlinear fractional Schr\"odinger equations on $\R^d$. These results extend the previous ones in \cite{HongSire} for $\sigma \geq 2$. This covers the well-known result for the Schr\"odinger equation $\sigma =2$ given in \cite{CazenaveWeissler}. In the case $\sigma \in (0,2)\backslash \{1\}$, we give the local well-posedness in sub-critical case for all exponent $\nu>1$ in contrast of ones in \cite{HongSire}. This also generalizes the ones of \cite{ChoHwangKwonLee} when $d=1$ and of \cite{GuoHuo} when $d\geq 2$ where the authors considered the cubic fractional Schr\"odinger equation with $\sigma \in (1,2)$. We also give the global existence in energy space under some assumptions. We finally prove the local well-posedness in sub-critical and critical cases for the pure power-type nonlinear fractional wave equations. 
\end{abstract}


\section{Introduction and main results}
\setcounter{equation}{0}
Let $\sigma \in (0,\infty)\backslash \{1\}$. We consider the Cauchy fractional Schr\"odinger and wave equations posed on $\R^d, d\geq 1$, namely
\begin{align}
\left\{
\begin{array}{cll}
i\partial_t u(t,x) - \Lambda^\sigma u(t,x) &=& -\mu (|u|^{\nu-1} u)(t,x), \quad (t, x) \in \R \times \R^d, \\
u(0,x) &=& \varphi(x), \quad x\in \R^d,
\end{array}
\right.
\tag{NLFS}
\end{align}
and
\begin{align}
\left\{
\begin{array}{cll}
\partial^2_t v(t,x) + \Lambda^{2\sigma} v(t,x) &=& -\mu (|v|^{\nu-1}v)(t,x), \quad (t, x) \in \R \times \R^d, \\
v(0,x) = \varphi(x),& & \partial_tv(0,x)= \phi(x), \quad x\in \R^d,
\end{array}
\right.
\tag{NLFW}
\end{align}
where $\nu>1$ and $\mu\in \{\pm 1\}$. The operator $\Lambda^\sigma=(\sqrt{-\Delta})^\sigma$ is the Fourier multiplier by $|\xi|^\sigma$ where $\Delta= \sum_{j=1}^d \partial^2_j$ is the free Laplace operator on $\R^d$. The number $\mu=1$ (resp. $\mu=-1$) corresponds to the defocusing case (resp. focusing case). When $\sigma \in (0,2)\backslash\{1\}$, the fractional Schr\"odinger equation was discovered by N. Laskin (see \cite{Laskin2000}, \cite{Laskin2002}) owing to the extension of the Feynman path integral, from the Brownian-like to L\'evy-like quantum mechanical paths. This equation also appears in the water wave models (see \cite{IonescuPusateri}). When $\sigma \in [2,\infty)$, the (NLFS) generalizes the nonlinear Schr\"odinger equation $\sigma=2$ or the fourth order nonlinear Schr\"odinger equation $\sigma =4$ (see e.g. \cite{Pausader} and references therein). In the case $\sigma \in (0,2)\backslash\{1\}$, the fractional wave equation, introduced in \cite{ChenHolm}, reflects the L\'evy stable process and the fractional Brownian motion. In the other side, when $\sigma \in [2,\infty)$, the (NLFW) can be seem as a generalization of the fourth order nonlinear wave equation (see e.g. \cite{Pausaderwave}).  \newline
\indent It is well known (see \cite{GinibreVelo84}, \cite{Kato95}, \cite{Cazenave} or \cite{Tao}) that the (NLFS) and (NLFW) are locally well-posed in $H^\gamma$ with $\gamma \geq d/2$ provided the nonlinearity is sufficiently regular. The main purpose of this note is to prove local well-posed results for (NLFS) and (NLFW) for $\gamma \in [0,d/2)$. For the (NLFS), we extend the previous results in \cite{HongSire} for $\sigma \geq 2$. This covers the well-known result for Schr\"odinger equation $\sigma=2$ given in \cite{Cazenave}. When $\sigma \in (0,2)\backslash \{1\}$, we show local well-posedness in sub-critical case for $\nu>1$ in contrast of $\nu\geq 2$ when $d=1$, $\nu \geq 3$ when $d\geq 2$ of \cite{HongSire}. This result generalizes the ones of \cite{ChoHwangKwonLee} when $d=1$ and of \cite{GuoHuo} when $d\geq 2$ where the authors considered the cubic fractional Schr\"odinger equation with $\sigma \in (1,2)$. We also shows the global existence in energy space, namely $H^{\sigma/2}$ under some assumptions. Moreover, in critical case, we prove global existence and scattering with small homogeneous data instead of inhomogeneous one in \cite{HongSire}. To our knowledge, the (NLFW) does not seem to have been much considered before, up to \cite{Wang} on scattering operator with $\sigma$ is an even integer  and \cite{ChenFanZhang14}, \cite{ChenFanZhang15} in the context of the damped fractional wave equation. \newline
\indent Let us introduce some standard notations (see \cite{GinibreVelo85}, Appendix, \cite{Triebel} or \cite{BerghLosfstom}). Let $\chi_0 \in C^\infty_0(\R^d)$ be such that $\chi_0(\xi)=1$ for $|\xi|\leq 1$ and $\text{supp}(\chi_0) \subset \{\xi \in \R^d, |\xi|\leq 2\}$. Set $\chi(\xi):= \chi_0(\xi)-\chi_0(2\xi)$. It is easy to see that $\chi\in C^\infty_0(\R^d)$ and $\text{supp}(\chi)\subset \{\xi \in \R^d, 1/2 \leq |\xi|\leq 2 \}$. We denote the Littlewood-Paley projections by $P_0:=\chi_0(D), P_N:=\chi(N^{-1}D)$ with $N=2^k, k\in \Z$ where $\chi_0(D), \chi(N^{-1}D)$ are the Fourier multipliers by $\chi_0(\xi)$ and $\chi(N^{-1}\xi)$ respectively. Given $\gamma \in \R$ and $1 \leq q \leq \infty$, one defines the Sobolev and Besov spaces as
\begin{align*}
H^\gamma_q &:= \Big\{ u \in \Sh' \ | \  \|u\|_{H^\gamma_q}:=\|\scal{\Lambda}^\gamma u\|_{L^q} <\infty \Big\}, \quad \scal{\Lambda}:=\sqrt{1+\Lambda^2}, \\
B^\gamma_q&:= \Big\{ u \in \Sh' \ | \ \|u\|_{B^\gamma_q}:= \|P_0 u\|_{L^q} + \Big( \sum_{N\in 2^{\N}} N^{2\gamma} \|P_N u\|^2_{L^q} \Big)^{1/2} <\infty \Big\},
\end{align*}
where $\Sh'$ is the space of tempered distributions. Now let $\Sh_0$ be a subspace of the Schwartz space $\Sh$ consisting of functions $\phi$ satisfying $D^\alpha \hat{\phi}(0)=0$ for all $\alpha \in \N^d$ where $\hat{\cdot}$ is the Fourier transform on $\Sh$ and $\Sh'_0$ its topology dual space. One can see $\Sh'_0$ as $\Sh'/\Ph$ where $\Ph$ is the set of all polynomials on $\R^d$. The homogeneous Sobolev and Besov spaces are defined by
\begin{align*}
\dot{H}^\gamma_q &:= \Big\{ u \in \Sh'_0 \ | \  \|u\|_{\dot{H}^\gamma_q}:=\|\Lambda^\gamma u\|_{L^q} <\infty \Big\}, \\
\dot{B}^\gamma_q &:= \Big\{ u \in \Sh'_0 \ | \  \|u\|_{\dot{B}^\gamma_q}:=\Big( \sum_{N\in 2^{\Z}} N^{2\gamma} \|P_N u\|^2_{L^q} \Big)^{1/2} <\infty \Big\}.
\end{align*}
It is easy to see that the norm $\|u\|_{B^\gamma_q}$ and $\|u\|_{\dot{B}^\gamma_q}$ does not depend on the choice of $\chi_0$, and $\Sh_0$ is dense in $\dot{H}^\gamma_q, \dot{B}^\gamma_q$. Under these settings, $H^\gamma_q, B^\gamma_q, \dot{H}^\gamma_q$ and $\dot{B}^\gamma_q$ are Banach spaces with the norms $\|u\|_{H^\gamma_q}, \|u\|_{B^\gamma_q}, \|u\|_{\dot{H}^\gamma_q}$ and $\|u\|_{\dot{B}^\gamma_q}$ respectively (see e.g. \cite{Triebel}). In the sequel, we shall use $H^\gamma:= H^\gamma_2$, $\dot{H}^\gamma:= \dot{H}^\gamma_2$. We note (see \cite{BerghLosfstom}, \cite{GinibreVelo85}) that if $2\leq q <\infty$, then $\dot{B}^\gamma_q \subset \dot{H}^\gamma_q$. The reverse inclusion holds for $1<r\leq 2$. In particular, $\dot{B}^\gamma_2= \dot{H}^\gamma$ and $\dot{B}^0_2=\dot{H}^0_2=L^2$. Moreover, if $\gamma>0$, then $H^\gamma_q = L^q \cap \dot{H}^\gamma_q$ and $B^\gamma_q = L^q \cap \dot{B}^\gamma_q$. \newline
\indent Before stating main results, we recall some facts on (NLFS) and (NLFW). By a standard approximation argument, we see that the following quantities are conserved by the flow of (NLFS),
\[
M_{\text{s}}(u)= \int |u(t,x)|^2 dx, \quad E_{\text{s}}(u)= \int \frac{1}{2}|\Lambda^{\sigma/2} u(t,x)|^2 + \frac{\mu}{\nu+1}|u(t,x)|^{\nu+1} dx.
\]
Moreover, if we set for $\lambda>0$, 
\[
u_\lambda(t,x)= \lambda^{-\frac{\sigma}{\nu-1}} u( \lambda^{-\sigma} t, \lambda^{-1} x),
\]
then the (NLFS) is invariant under this scaling that is for $T \in (0, +\infty]$,
\[
u \text{ solves (NLFS) on } (-T, T) \Longleftrightarrow u_\lambda \text{ solves (NLFS) on } (-\lambda^{\sigma}T, \lambda^{\sigma}T).
\]
We also have
\[
\|u_\lambda(0)\|_{\dot{H}^\gamma} = \lambda^{\frac{d}{2}-\frac{\sigma}{\nu-1}- \gamma} \|\varphi\|_{\dot{H}^\gamma}. \nonumber
\]
From this, we define the critical regularity exponent for (NLFS) by
\begin{align}
\gamma_{\text{s}} =\frac{d}{2} - \frac{\sigma}{\nu-1}. \label{critical exponent schrodinger}
\end{align}
Similarly, the following energy is conserved under the flow of (NLFW),
\[
E_{\text{w}}(v)=\int\frac{1}{2} |\partial_t v (t,x)|^2 + \frac{1}{2}|\Lambda^{\sigma} v(t,x)|^2 + \frac{\mu}{\nu+1}|v(t,x)|^{\nu+1} dx,
\]
and the (NLFW) is invariant under the following scaling
\[
v_\lambda(t,x)= \lambda^{-\frac{2\sigma}{\nu-1}} v( \lambda^{-\sigma} t, \lambda^{-1} x).
\]
Using that
\begin{align}
\|v_\lambda(0)\|_{\dot{H}^\gamma} &= \lambda^{\frac{d}{2}-\frac{2\sigma}{\nu-1}- \gamma} \|\varphi\|_{\dot{H}^\gamma}, \nonumber \\
\|\partial_tv_\lambda(0)\|_{\dot{H}^{\gamma-\sigma}} &= \lambda^{\frac{d}{2}-\frac{2\sigma}{\nu-1}- \gamma} \|\phi\|_{\dot{H}^{\gamma-\sigma}}, \nonumber
\end{align}
we define the critical regularity exponent for (NLFW) by
\begin{align}
\gamma_{\text{w}} =\frac{d}{2} - \frac{2\sigma}{\nu-1}. \label{critical exponent wave}
\end{align}
By the standard argument (see e.g \cite{LinSog}), it is easy to see that the (NLFS) (resp. (NLFW)) is ill-posed if $\varphi \in \dot{H}^\gamma$ with $\gamma<\gamma_{\text{s}}$ (resp. $v_0 \in \dot{H}^\gamma, v_1 \in \dot{H}^{\gamma-\sigma}$ with $\gamma < \gamma_{\text{w}}$). Indeed if $u$ solves the (NLFS) with initial data $\varphi \in \dot{H}^\gamma$ with the lifespan $T$, then the norm $\|u_\lambda(0)\|_{\dot{H}^\gamma}$ and the lifespan of $u_\lambda$ go to zero with $\lambda$. Thus we can expect the well-posed results for (NLFS) (resp. (NLFW)) when $\gamma \geq \gamma_{\text{s}}$ (resp. $\gamma \geq \gamma_\text{w}$). \newline
\indent Throughout this note, a pair $(p,q)$ is said to be admissible if 
\begin{align}
(p,q) \in [2,\infty]^2, \quad (p,q,d) \ne (2,\infty,2), \quad \frac{2}{p}+\frac{d}{q} \leq \frac{d}{2}. \label{admissible condition}
\end{align}
We also denote for $(p,q) \in [1,\infty]^2$,
\begin{align}
\gamma_{p,q}= \frac{d}{2}-\frac{d}{q} - \frac{\sigma}{p}. \label{define gamma}
\end{align} 
Note that when $\sigma \in (0,2) \backslash \{1\}$, then $\gamma_{p,q}>0$ for all admissible pair except $(p,q)=(\infty, 2)$. Therefore, it convenient to separate two cases $\sigma \in (0,2)\backslash \{1\}$ and $\sigma \in [2,\infty)$. Our first result is the following local well-posedness for (NLFS) in sub-critical case.  
\begin{theorem} \label{theorem local wellposedness subcritical schrodinger sigma <2}
Given $\sigma \in (0,2)\backslash \{1\}$ and $\nu >1$. Let $\gamma \in [0,d/2)$ be such that
\begin{align}
\left\{
\begin{array}{ll}
\gamma > 1/2 - \sigma/\max ( \nu-1, 4) & \text{when } d=1, \\
\gamma > d/2 - \sigma/\max ( \nu-1, 2) & \text{when } d \geq 2,
\end{array}
\right. \label{condition local wellposedness subcritical schrodinger}
\end{align}
and also, if $\nu$ is not an odd integer, 
\begin{align}
\lceil\gamma \rceil \leq \nu, \label{assumption smoothness nonlinearity}
\end{align}
where $\lceil\gamma\rceil$ is the smallest positive integer greater than or equal to $\gamma$. Then for all $\varphi \in H^\gamma$, there exist $T^*\in (0,\infty]$ and a unique solution to $\emph{(NLFS)}$ satisfying 
\[
u \in C([0,T^*), H^\gamma) \cap L^p_{\emph{loc}}([0,T^*), L^\infty),
\] 
for some $p > \max (\nu-1,4)$ when $d=1$ and some $p >\max(\nu-1, 2)$ when $d\geq 2$. Moreover, the following properties hold:
\begin{itemize}
\item[\emph{i.}] If $T^*< \infty$, then $\|u(t)\|_{H^\gamma} = \infty$ as $t \rightarrow T^*$.
\item[\emph{ii.}] $u$ depends continuously on $\varphi$ in the following sense. There exists $0< T< T^*$ such that if $\varphi_n \rightarrow \varphi$ in $H^\gamma$ and if $u_n$ denotes the solution of \emph{(NLFS)} with initial data $\varphi_n$, then $0<T< T^*(\varphi_n)$ for all $n$ sufficiently large and $u_n$ is bounded in $L^a([0,T],H^{\gamma-\gamma_{a,b}}_b)$ for any admissible pair $(a,b)$ with $b<\infty$. Moreover, $u_n \rightarrow u$ in $L^a([0,T],H^{-\gamma_{a,b}}_b)$ as $n \rightarrow \infty$. In particular, $u_n \rightarrow u$ in $C([0,T],H^{\gamma-\ep})$ for all $\ep>0$.
\end{itemize}
\end{theorem}
\begin{rem} \label{rem local wellposedness subcritical schrodinger}
\begin{itemize}
\item[i.] The proof of this result is based on Strichartz estimates and the fractional derivatives of nonlinear operators (see Section $\ref{section nonlinear estimates}$). Note that when $\nu$ is an odd integer, $F(\cdot)=-\mu|\cdot|^{\nu-1} \cdot \in C^\infty(\C, \C)$ (in the real sense) and when $\nu$ is not an odd integer satisfying $(\ref{assumption smoothness nonlinearity})$, $F \in C^{\lceil\gamma\rceil}(\C, \C)$. Thus we are able to apply the fractional derivatives up to order $\gamma$ (see Corollary $\ref{coro fractional derivatives}$).
\item[ii.] If we assume that $\nu>1$ is an odd integer or 
\begin{align}
\lceil \gamma \rceil \leq \nu-1 \label{assumption continuity}
\end{align}
otherwise, then the continuous dependence holds in $C([0,T],H^\gamma)$ (see Remark $\ref{rem continuity in H gamma}$). 
\end{itemize}
\end{rem}
As mentioned before, this result improves the one in \cite{HongSire} at the point that Hong-Sire only give the local well-posedness for $\nu \geq 2$ when $d =1$ and $\nu \geq 3$ when $d \geq 2$. This result also covers the ones \cite{ChoHwangKwonLee} and \cite{GuoHuo} where the authors consider the cubic fractional Schr\"odinger equation with $\sigma \in (1,2)$. When $\sigma \geq 2$, we have the following better result which generalizes the case $\sigma =2$ given in \cite{CazenaveWeissler}. 
\begin{theorem} \label{theorem local wellposedness subcritical schrodinger sigma geq 2}
Given $\sigma \geq 2$ and $\nu >1$. Let $\gamma \in [0,d/2)$ be such that $\gamma>\gamma_\emph{s}$, and also, if $\nu$ is not an odd integer, $(\ref{assumption smoothness nonlinearity})$. Let $(p,q)$ be the admissible pair defined by
\begin{align}
p=\frac{2\sigma(\nu+1)}{(\nu-1)(d-2\gamma)}, \quad q=\frac{d(\nu+1)}{d+(\nu-1)\gamma}. \label{define p q}
\end{align}
Then for all $\varphi \in H^\gamma$, there exist $T^* \in (0,\infty]$ and a unique solution to \emph{(NLFS)} satisfying
\[
u \in C([0,T^*),H^\gamma) \cap L^p_{\emph{loc}}([0,T^*), H^{\gamma}_q).
\] 
Moreover, the following properties hold:
\begin{itemize}
\item[\emph{i.}] If $T^*< \infty$, then $\|u(t)\|_{\dot{H}^\gamma} = \infty$ as $t \rightarrow T^*$.
\item[\emph{ii.}] $u$ depends continuously on $\varphi$ in the following sense. There exists $0< T< T^*$ such that if $\varphi_n \rightarrow \varphi$ in $H^\gamma$ and if $u_n$ denotes the solution of \emph{(NLFS)} with initial data $\varphi_n$, then $0<T< T^*(\varphi_n)$ for all $n$ sufficiently large and $u_n$ is bounded in $L^a([0,T],H^\gamma_b)$ for any admissible pair $(a,b)$ with $\gamma_{a,b}=0$ and $b<\infty$. Moreover, $u_n \rightarrow u$ in $L^a([0,T],L^b)$ as $n \rightarrow \infty$. In particular, $u_n \rightarrow u$ in $C([0,T],H^{\gamma-\ep})$ for all $\ep>0$.
\end{itemize}
\end{theorem}
Thanks to the conservation of mass, we have immediately the following global well-posedness in $L^2$ when $\sigma \geq 2$.
\begin{coro} \label{coro global wellposedness in L2}
Let $\sigma \geq 2$ and $\nu \in (1, 1+ 2\sigma/d)$. Then for all $\varphi \in L^2$, there exists a unique global solution $u \in C(\R, L^2(\R^d))$ to \emph{(NLFS)}.
\end{coro}
\begin{prop} \label{prop global subcritical schrodinger}
Let 
\begin{align}
\left\{
\begin{array}{ll}
\sigma \in (2/3,1) & \text{when } d=1, \\
\sigma \in (1,2) & \text{when } d=2, \\
\sigma \in (3/2, 3) & \text{when } d=3, \\
\sigma \in [2, d) & \text{when } d \geq 4,
\end{array}
\right. \label{assumption global existence}
\end{align}
and $\nu>1$ be such that $\sigma/2>\gamma_\emph{s}$, and also, if $\nu$ is not an odd integer, $\lceil\sigma/2\rceil \leq \nu$. Then for any $\varphi \in H^{\sigma/2}$, the solution to \emph{(NLFS)} given in \emph{Theorem } $\ref{theorem local wellposedness subcritical schrodinger sigma <2}$ and \emph{Theorem } $\ref{theorem local wellposedness scattering schrodinger sigma geq 2}$ can be extended to the whole $\R$ if one of the following is satisfied:
\begin{itemize}
\item[\emph{i.}] $\mu=1$.
\item[\emph{ii.}] $\mu =-1, \nu <1+2\sigma/d$.
\item[\emph{iii.}] $\mu =-1, \nu=1+2\sigma/d$ and $\|\varphi\|_{L^2}$ is small.
\item[\emph{iv.}] $\mu=-1$ and $\|\varphi\|_{H^{\sigma/2}}$ is small.
\end{itemize}
\end{prop}
\indent We now turn to the local well-posedness and scattering with small data for (NLFS) in critical case. 
\begin{theorem} \label{theorem local wellposedness scattering critical schrodinger sigma <2}
Let $\sigma \in (0,2)\backslash \{1\}$ and 
\begin{align} \label{condition scattering schrodinger}
\left\{
\begin{array}{l}
\nu>5 \text{ when } d=1, \\
\nu>3 \text{ when } d\geq 2
\end{array}
\right.
\end{align}
be such that $\gamma_{\emph{s}} \geq 0$, and also, if $\nu$ is not an odd integer,
\begin{align}
\lceil\gamma_{\emph{s}}\rceil \leq \nu. \label{assumption smoothness critical}
\end{align} 
Then for all $\varphi \in H^{\gamma_{\emph{s}}}$, there exist $T^* \in (0,\infty]$ and a unique solution to $\emph{(NLFS)}$ satisfying 
\[
u \in C([0,T^*), H^{\gamma_{\emph{s}}}) \cap L^p_\emph{loc}([0,T^*),B^{\gamma_{\emph{s} }-\gamma_{p,q}}_{q}),
\] 
where $p=4, q=\infty$ when $d=1$; $2<p<\nu-1, q=p^\star=2p/(p-2)$ when $d=2$ and $p=2, q=2^\star=2d/(d-2)$ when $d\geq 3$. 
Moreover, if $\|\varphi\|_{\dot{H}^{\gamma_{\emph{s}}}} < \varepsilon$ for some $\varepsilon>0$ small enough, then $T^*=\infty$ and the solution is scattering in $H^{\gamma_{\emph{s}}}$, i.e. there exists $\varphi^+ \in H^{\gamma_\emph{s}}$ such that
\[
\lim_{t\rightarrow+\infty} \|u(t)-e^{-it\Lambda^\sigma} \varphi^+\|_{H^{\gamma_\emph{s}}} =0. 
\]
\end{theorem}
This theorem is just a slightly modification of Theorem 1.2 and Theorem 1.3 in \cite{HongSire} where the authors proved the global existence and scattering for small inhomogeneous data. Note that Strichartz estimate is not sufficient to give the local existence in critical case. It needs a delicate estimate on $L^{\nu-1}_tL^\infty_x$ (see Lemma 3.5 in \cite{HongSire}). The range $\nu \in (1,5]$ when $d=1$ and $\nu \in (1,3]$ still remains open, and it requires another technique rather than Strichartz estimate. The situation becomes better when $\sigma \geq 2$, and we have the following result.
\begin{theorem} \label{theorem local wellposedness scattering schrodinger sigma geq 2}
Let $\sigma \geq 2$ and $\nu>1$ such that $\gamma_{\emph{s}} \geq 0$, and also, if $\nu$ is not an odd integer, $(\ref{assumption smoothness critical})$. Let 
\begin{align}
p= \nu+1, \quad q = \frac{2d(\nu+1)}{d(\nu+1)-2\sigma}. \label{define p q critical}
\end{align}
Then for any $\varphi \in H^{\gamma_\emph{s}}$, there exist $T^*\in (0,\infty]$ and a unique solution to \emph{(NLFS)} satisfying
\[
u \in C([0,T^*),H^{\gamma_\emph{s}}) \cap L^p_{\emph{loc}}([0,T^*),H^{\gamma_\emph{s}}_q).
\] 
Moreover, if $\|\varphi\|_{\dot{H}^{\gamma_{\emph{s}}}} < \varepsilon$ for some $\varepsilon>0$ small enough, then $T^*=\infty$ and the solution is scattering in $H^{\gamma_{\emph{s}}}$.
\end{theorem}
We now give the local well-posed results for the (NLFW) equation. Let us start with the local well-posedness in sub-critical case.
\begin{theorem} \label{theorem local wellposedness subcritical wave}
Given $\sigma \in (0,\infty)\backslash \{1\}$ and $\nu >1$. Let $\gamma \in [0,d/2)$ be as in $(\ref{condition local wellposedness subcritical schrodinger})$ and also, if $\nu$ is not an odd integer, $(\ref{assumption smoothness nonlinearity})$.
Then for all $(\varphi,\phi) \in H^\gamma\times H^{\gamma-\sigma}$, there exist $T^* \in (0,\infty]$ and a unique solution to \emph{(NLFW)} satisfying 
\[
v \in C([0,T^*), H^\gamma) \cap C^1([0,T^*), H^{\gamma-\sigma}) \cap L^p_{\emph{loc}}([0,T^*), L^\infty),
\] 
for some $p > \max (\nu-1,4)$ when $d=1$ and some $p >\max(\nu-1, 2)$ when $d\geq 2$. Moreover, the following properties hold:
\begin{itemize}
\item[\emph{i.}] If $T^*<\infty$, then $\|[v](t)\|_{H^\gamma} =\infty$ as $t\rightarrow T^*$.
\item[\emph{ii.}] $v$ depends continuously on $(\varphi, \phi)$ in the following sense. There exists $0< T<T^*$ such that if $(\varphi_n,\phi_n) \rightarrow (\varphi,\phi)$ in $H^\gamma \times H^{\gamma-\sigma}$ and if $v_n$ denotes the solution of \emph{(NLFW)} with initial data $(\varphi_n, \phi_n)$, then $0<T< T^*(\varphi_n, \phi_n)$ for all $n$ sufficiently large and $v_n$ is bounded in $L^a([0,T],H^{\gamma-\gamma_{a,b}}_b)$ for any admissible pair $(a,b)$ with $b<\infty$. Moreover, $v_n \rightarrow v$ in $L^a(I,H^{-\gamma_{a,b}}_b)$ as $n\rightarrow \infty$. In particular, $v_n\rightarrow v$ in $C([0,T],H^{\gamma-\ep}) \cap C^1([0,T],H^{\gamma-\sigma-\ep})$ for all $\ep>0$.  
\end{itemize}
\end{theorem}
We note that $(\ref{condition local wellposedness subcritical schrodinger})$ is necessary to use the Sobolev embedding, but it produces a gap between $\gamma_{\text{w}}$ and $1/2-\sigma/\max ( \nu-1, 4)$ when $d=1$ and $d/2-\sigma/\max ( \nu-1, 2)$ when $d\geq 2$. Moreover, if we assume that $\nu>1$ is an odd integer or $(\ref{assumption continuity})$ otherwise, then the continuous dependence holds in $C([0,T],H^\gamma)\cap C^1([0,T],H^{\gamma-\sigma})$. \newline
\indent The following result gives the local well-posedness for (NLFW) in $\sigma$-sub-critical case.
\begin{theorem} \label{theorem local wellposedness wave H sigma subcritical}
1. Assume for $d=1,2,3,4$,
\begin{multline}
\sigma \in \Big(0,\frac{d}{d+2}\Big),\  \nu\in \Big(\frac{d}{d-2\sigma}, \frac{2d-d\sigma}{2d-(d+4)\sigma}\Big] \\ 
\text{ or } \sigma \in \Big[\frac{d}{d+2}, \frac{d}{2}\Big) \backslash \{1\}, \  \nu \in \Big(\frac{d}{d-2\sigma}, \frac{d+2\sigma}{d-2\sigma}\Big); \label{assumption sigma <2 subcritical 1}
\end{multline}
for $d=5,...,11$,
\begin{multline}
\sigma \in \Big(0,\frac{2}{3}\Big), \ \nu\in \Big(\frac{d}{d-2\sigma}, \frac{2d-d\sigma}{2d-(d+4)\sigma}\Big] \text{ or } \sigma \in \Big[\frac{2}{3}, \frac{d}{6}\Big) \backslash\{1\},\  \nu \in \Big(\frac{d}{d-2\sigma}, \frac{d}{d-3\sigma}\Big] \\
\text{ or } \sigma  \in \Big[\frac{d}{6},2\Big)\backslash\{1\}, \  \nu \in \Big(\frac{d}{d-2\sigma}, \frac{d+2\sigma}{d-2\sigma}\Big); \label{assumption sigma <2 subcritical 2}
\end{multline}
and for $d\geq 12$,
\begin{multline}
\sigma \in \Big(0,\frac{2}{3}\Big), \ \nu \in \Big(\frac{d}{d-2\sigma}, \frac{2d-d\sigma}{2d-(d+4)\sigma}\Big] \text{ or } \sigma \in \Big[\frac{2}{3}, 2\Big) \backslash\{1\}, \ \nu \in \Big(\frac{d}{d-2\sigma}, \frac{d}{d-3\sigma}\Big]. \label{assumption sigma <2 subcritical 3}
\end{multline}
Let $(p,q)$ be an admissible pair defined by
\begin{align}
p=\frac{2\sigma\nu}{(d-2\sigma)\nu-d}, \quad q=2\nu. \label{define p q sigma subcritical wave sigma <2}
\end{align}
Then for all $(\varphi, \phi) \in \dot{H}^\sigma \times L^2$, there exist $T^* \in (0,\infty]$ and a unique solution to \emph{(NLFW)} satisfying 
\[
v \in C([0,T^*),\dot{H}^\sigma) \cap C^1([0,T^*),L^2) \cap L^p_\emph{loc}([0,T^*),L^q).
\] 
Moreover, the following properties hold:
\begin{itemize}
\item[\emph{i.}] If $T^*<\infty$, then $\|[v](t)\|_{\dot{H}^\sigma} =\infty$ as $t\rightarrow T^*$. 
\item[\emph{ii.}] $v$ depends continuously on $(\varphi, \phi)$ in the following sense. There exists $0< T<T^*$ such that if $(\varphi_n,\phi_n) \rightarrow (\varphi,\phi)$ in $\dot{H}^\sigma \times L^2$ and if $v_n$ denotes the solution of \emph{(NLFW)} with initial data $(\varphi_n, \phi_n)$, then $0<T< T^*(\varphi_n, \phi_n)$ for all $n$ sufficiently large and $v_n\rightarrow v$ in $C([0,T],\dot{H}^{\sigma}) \cap C^1([0,T],L^2)$.  
\end{itemize}
\indent 2. Let 
\begin{align}
\sigma \in \Big[2,\frac{d}{2}\Big), \quad \nu \in \Big[\frac{d\sigma^*}{d+\sigma} , \sigma^*\Big), \label{assumption sigma subcritical sigma geq 2}
\end{align}
where $\sigma^*:= (d+2\sigma)/(d-2\sigma)$. Let $(p,q)$ be an admissible pair defined by
\begin{align}
p=2\sigma^*, \quad p=\frac{2d\sigma^*}{d+\sigma}. \label{define p q sigma subcritical wave sigma geq 2}
\end{align}
Then the same conclusion as in \emph{Item 1} holds true.
\end{theorem}
This theorem and the conservation of energy imply the following global well-posedness for the defocusing (NLFW). 
\begin{coro} \label{coro global wellposedness wave H sigma subcritical defocusing}
Under assumptions of $\emph{Theorem}$ $\ref{theorem local wellposedness wave H sigma subcritical}$, for all $(\varphi, \phi) \in \dot{H}^\sigma\times L^2$, there exists a unique global solution to the defocusing $(\emph{NLFW})$ satisfying 
\[
v \in C(\R,\dot{H}^\sigma) \cap C^1(\R,L^2) \cap L^p_{\emph{loc}}(\R,L^q),
\] 
where $(p,q)$ are as in \emph{Theorem } $\ref{theorem local wellposedness wave H sigma subcritical}$. 
\end{coro}
The following result gives the local well-posedness with small data scattering for (NLFW) in critical cases.
\begin{theorem} \label{theorem local wellposedness critical scattering small data wave}
1. Assume for $d \geq 1$ that
\begin{align}
\sigma \in \Big[\frac{d}{d+1},d\Big) \backslash \{1\}, \quad \nu \in \Big[1+ \frac{4\sigma}{d-\sigma}, \infty \Big), \label{assumption critical wave 1} 
\end{align}
and also, if $\nu$ is not an odd integer, 
\begin{align}
\lceil\gamma_{\text{w}}\rceil-\frac{\sigma}{2} \leq \nu-1. \label{assumption smoothness critical wave}
\end{align}
Let $p, a$ be defined by
\begin{align}
p= \frac{(d+\sigma)(\nu-1)}{2\sigma}, \quad a =\frac{2(d+\sigma)}{d-\sigma}. \label{define p a}
\end{align}
Then for all $(\varphi,\phi) \in \dot{H}^{\gamma_{\emph{w}}}\times\dot{H}^{\gamma_{\emph{w}}-\sigma}$, there exist $T^* \in (0,\infty]$ and a unique solution to \emph{(NLFW)} satisfying
\[
v \in C([0,T^*),\dot{H}^{\gamma_{\emph{w}}}) \cap C^1([0,T^*),\dot{H}^{\gamma_{\emph{w}}-\sigma}) \cap L^p_\emph{loc}([0,T^*), L^p) \cap L^a_\emph{loc}([0,T^*),\dot{H}^{\gamma_{\emph{w}}-\frac{\sigma}{2}}_a).
\]
Moreover, if $\|[v](0)\|_{\dot{H}^{\gamma_{\text{w}}}} < \varepsilon$ for some $\varepsilon>0$ small enough, then $T^* =\infty$ and the solution is scattering in $\dot{H}^{\gamma_{\emph{w}}} \times  \dot{H}^{\gamma_{\emph{w}}-\sigma}$, i.e. there exist $(\varphi^+,\phi^+) \in \dot{H}^{\gamma_{\emph{w}}}\times \dot{H}^{\gamma_{\emph{w}}-\sigma}$ such that the (weak) solution to the linear fractional wave equation 
\[
\left\{
\begin{array}{ccl}
\partial^2_t v^+(t,x) + \Lambda^{2\sigma} v^+(t,x) &=& 0, \quad (t,x) \in \R \times \R^d, \\
v^+(0,x)= \varphi^+(x), & & \partial_t v^+ (0,x) =\phi^+(x), \quad x\in \R^d,
\end{array}
\right. 
\]
satisfy
\[
\lim_{t \rightarrow + \infty} \|[v(t)-v^+(t)]\|_{\dot{H}^{\gamma_{\emph{w}}}} =0.
\]
\indent 2. Assume for $d \geq 1$ that 
\begin{multline}
\sigma \in \Big[\frac{d^2+4d}{3d+4}, \infty\Big) \backslash \{1\}, \quad \nu \in \Big[ 1 + \frac{4\sigma(d+2)}{d(d+\sigma)}, \infty \Big) \\
\emph{ or } \sigma \in \Big[\frac{d}{d+1}, \frac{d^2+4d}{3d+4} \Big) \backslash \{1\}, \quad \nu \in \Big[1+ \frac{4\sigma (d+2)}{d(d+\sigma)}, 1+ \frac{4\sigma(d+2)}{d^2-3d\sigma +4d-4\sigma}\Big]. \label{assumption critical wave 2}
\end{multline}
Then for all $(\varphi,\phi) \in \dot{H}^{\gamma_{\emph{w}}}\times \dot{H}^{\gamma_{\emph{w}}-\sigma}$, there exist $T^*\in (0,\infty]$ and a unique solution to \emph{(NLFW)} satisfying
\[
v \in C([0,T^*),\dot{H}^{\gamma_{\emph{w}}}) \cap C^1([0,T^*),\dot{H}^{\gamma_{\emph{w}}-\sigma}) \cap L^p_\emph{loc}([0,T^*),L^p),
\]
where $p$ is as above. Moreover, if $\|[v](0)\|_{\dot{H}^{\gamma_{\emph{w}}}} < \varepsilon$ for some $\varepsilon>0$ small enough, then $T^* =\infty$ and the solution is scattering in $\dot{H}^{\gamma_{\emph{w}}} \times \dot{H}^{\gamma_{\emph{w}}-\sigma}$.
\end{theorem}
Finally, we have the following local well-posedness and scattering with small data for (NLFW) in $\sigma$-critical case.
\begin{theorem} \label{theorem local wellposedness wave H sigma critical}
Let 
\begin{align}
\left\{
\begin{array}{cl}
\sigma \in \Big[\frac{d}{d+2}, \frac{d}{2}\Big) \backslash \{1\} & \text{when } d=\{1,2,3,4\}, \\
\sigma \in \Big[\frac{d}{6}, \frac{d}{2}\Big) \backslash \{1\} & \text{when } d \geq 5,
\end{array}
\right. \label{H sigma critical wave condition}
\end{align}
and $\nu = 1+ 4 \sigma/(d-2\sigma)$. Then for all $(\varphi, \phi) \in \dot{H}^\sigma \times L^2$, there exist $T^* \in (0,\infty]$ and a unique solution to \emph{(NLFW)} satisfying 
\[ 
v \in C([0,T^*),\dot{H}^\sigma) \cap C^1([0,T^*),L^2) \cap L^{\nu}_\emph{loc}([0,T^*),L^{2\nu}).
\] 
Moreover, if $\|[v](0)\|_{\dot{H}^\sigma} < \varepsilon$ for some $\varepsilon>0$ small enough, then $T^*=\infty$ and the solution is scattering in $\dot{H}^\sigma \times L^2$.
\end{theorem}
\indent The rest of this note is organized as follows. In Section 2, we prove Strichartz estimates for the fractional Schr\"odinger and wave equations. In Section 3, we recall the fractional derivatives of the nonlinearity. Section 4 is devoted to the proofs of local well-posedness for sub-critical (NLFS) and the local well-posedness with small data scattering for critical (NLFS). We finally prove the local well-posedness for sub-critical (NLFW) and the local well-posedness with small data scattering for critical (NLFW) in Section 5. 
\section{Strichartz estimates}
\setcounter{equation}{0}
In this section, we recall Strichartz estimates for the linear fractional Schr\"odinger and wave equations. 
\begin{theorem}[Strichartz estimates \cite{ChoOzawaXia}] \label{theorem full strichartz schrodinger} 
Let $d\geq 1, \sigma \in (0,\infty)\backslash \{1\}, \gamma\in \R$ and a (weak) solution to the linear fractional Schr\"odinger equation, namely
\[
u(t)= e^{-it\Lambda^\sigma} \varphi + \int_{0}^{t} e^{-i(t-s)\Lambda^\sigma} F(s) ds,
\]
for some data $\varphi, F$. Then for all $(p,q)$ and $(a,b)$ admissible pairs, 
\begin{align}
\|u\|_{L^p(\R,\dot{B}^\gamma_q)} \lesssim \|\varphi\|_{\dot{H}^{\gamma+\gamma_{p,q}}} + \|F\|_{L^{a'}(\R,\dot{B}^{\gamma+\gamma_{p,q}-\gamma_{a',b'}-\sigma}_{b'})}, \label{full strichartz schrodinger}
\end{align}
where $\gamma_{p,q}$ and  $\gamma_{a',b'}$ are as in $(\ref{define gamma})$. In particular, 
\begin{align}
\|u\|_{L^p(\R,\dot{B}^{\gamma-\gamma_{p,q}}_q)} \lesssim \|\varphi\|_{\dot{H}^{\gamma}} + \|F\|_{L^1(\R,\dot{H}^\gamma)}, \label{homogeneous full strichartz schrodinger}
\end{align}
and 
\begin{align}
\|u\|_{L^\infty(\R,\dot{B}^{\gamma_{p,q}}_2)}+\|u\|_{L^p(\R,\dot{B}^0_q)} \lesssim \|\varphi\|_{\dot{H}^{\gamma_{p,q}}} + \|F\|_{L^{a'}(\R,\dot{B}^0_{b'})}, \label{inhomogeneous full strichartz schrodinger}
\end{align}
provided that 
\begin{align}
\gamma_{p,q}=\gamma_{a',b'}+\sigma. \label{gap condition}
\end{align}
Here $(a,a')$ is a conjugate pair.
\end{theorem}
\noindent \textit{Sketch of proof.} We firstly note this theorem is proved if we establish
\begin{align}
\|e^{-it\Lambda^\sigma}P_1 \varphi\|_{L^p(\R, L^q)} &\lesssim  \|P_1 \varphi\|_{L^2}, \label{strichartz reduction 1} \\
\Big|\Big|\int_0^t e^{-i(t-s)\Lambda^\sigma} P_1 F(s)ds \Big| \Big|_{L^p(\R, L^q)} &\lesssim  \|P_1 F\|_{L^{a'}(\R, L^{b'})}, \label{strichartz reduction 2}
\end{align}
for all $(p,q)$, $(a,b)$ admissible pairs. Indeed, by change of variables, we see that
\begin{align}
\|e^{-it\Lambda^\sigma} P_N \varphi \|_{L^p(\R, L^q)} &= N^{-(d/q+\sigma/p)} \|e^{-it\Lambda^\sigma} P_1   \varphi_N\|_{L^p(\R, L^q)}, \nonumber \\
\|P_1 \varphi_N\|_{L^2} &= N^{d/2} \|P_N \varphi\|_{L^2}, \nonumber
\end{align}
where $\varphi_N(x)= \varphi(N^{-1}x)$. The estimate $(\ref{strichartz reduction 1})$ implies that
\begin{align}
\|e^{-it\Lambda^\sigma}P_N  \varphi\|_{L^p(\R, L^q)}\lesssim N^{\gamma_{p,q}} \|P_N \varphi\|_{L^2}, \label{homogeneous strichartz reduction}
\end{align}
for all $N \in 2^{\Z}$. Similarly,
\begin{multline}
\Big|\Big|\int_0^t e^{-i(t-s)\Lambda^\sigma} P_N F(s)ds \Big|\Big|_{L^p(\R, L^q)} \nonumber \\ 
= N^{-(d/q+\sigma/p+\sigma)} \Big| \Big| \int_0^t e^{-i(t-s)\Lambda^\sigma} P_1  F_N(s)ds\Big|\Big|_{L^p(\R,L^q)}, \nonumber  
\end{multline}
where $F_N(t,x)= F(N^{-\sigma}t,N^{-1}x)$. We also have from $(\ref{strichartz reduction 2})$ and the fact 
\[
\|P_1 F_N\|_{L^{a'}(\R, L^{b'})} = N^{(d/b' + \sigma/a')} \|P_NF\|_{L^{a'}(\R, L^{b'})}
\]
that
\begin{align}
\Big|\Big|\int_0^t e^{-i(t-s)\Lambda^\sigma} P_N F(s)ds \Big| \Big|_{L^p(\R, L^q)} \lesssim N^{\gamma_{p,q}-\gamma_{a',b'}-\sigma}\|P_N F\|_{L^{a'}(\R, L^{b'})}, \label{inhomogeneous strichartz reduction}
\end{align}
for all $N \in 2^{\Z}$. We see from $(\ref{homogeneous strichartz reduction})$ and $(\ref{inhomogeneous strichartz reduction})$ that
\begin{align}
N^{\gamma}\|P_Nu\|_{L^p(\R, L^q)} \lesssim N^{\gamma+\gamma_{p,q}} \|P_N \varphi\|_{L^2} + N^{\gamma+\gamma_{p,q}-\gamma_{a',b'}-\sigma} \|P_N F\|_{L^{a'}(\R,L^{b'})}. \nonumber
\end{align}
By taking the $\ell^2(2^{\Z})$ norm both sides and using the Minkowski inequality, we get $(\ref{full strichartz schrodinger})$. The estimates $(\ref{homogeneous full strichartz schrodinger})$ and $(\ref{inhomogeneous full strichartz schrodinger})$ follow easily from $(\ref{full strichartz schrodinger})$. It remains to prove $(\ref{strichartz reduction 1})$ and $(\ref{strichartz reduction 2})$. By the $TT^*$-criterion (see \cite{KeelTaoTTstar} or \cite{BCDfourier}), we need to show
\begin{align}
\|T(t)\|_{L^2 \rightarrow L^2} &\lesssim  1, \label{energy estimate} \\
\|T(t)\|_{L^1 \rightarrow L^\infty} &\lesssim (1+ |t|)^{-d/2}, \label{dispersive estimate}
\end{align}
for all $t \in \R$ where $T(t):= e^{-it\Lambda^\sigma} P_1$. The energy estimate $(\ref{energy estimate})$ is obvious by using the Plancherel theorem. The dispersive estimate $(\ref{dispersive estimate})$ follows by the standard stationary phase theorem. The proof is complete.
\defendproof
\begin{coro} \label{coro usual strichartz schrodinger}
Let $d\geq 1, \sigma \in (0,\infty)\backslash \{1\}, \gamma \in \R$. If $u$ is a (weak) solution to the linear fractional Schr\"odinger equation for some data $\varphi, F$, then for all $(p,q)$ and $(a,b)$ admissible with $q<\infty$ and $b<\infty$ satisfying $(\ref{gap condition})$, 
\begin{align}
\|u\|_{L^p(\R,\dot{H}^{\gamma-\gamma_{p,q}}_q)} &\lesssim \|\varphi\|_{\dot{H}^\gamma} + \|F\|_{L^1(\R, \dot{H}^\gamma)}, \label{usual homogeneous strichartz schrodinger} \\
\|u\|_{L^\infty(\R,\dot{H}^{\gamma_{p,q}})}+\|u\|_{L^p(\R, L^q)} &\lesssim \|\varphi\|_{\dot{H}^{\gamma_{p,q}}} + \|F\|_{L^{a'}(\R,L^{b'})}. \label{usual inhomogeneous strichartz schrodinger}
\end{align}
\end{coro}
\begin{coro} \label{coro local strichartz schrodinger}
Let $d \geq 1$, $\sigma \in (0,\infty)\backslash \{1\}$, $\gamma \geq 0$ and $I$ a bounded interval. If $u$ is a (weak) solution to the linear fractional Schr\"odinger equation for some data $\varphi, F$, then for all $(p,q)$ admissible satisfying $q <\infty$,
\begin{align}
\|u\|_{L^p(I,H^{\gamma-\gamma_{p,q}}_q)} \lesssim \|\varphi\|_{H^\gamma} + \|F\|_{L^1(I, H^\gamma)}. \label{local strichartz schrodinger} 
\end{align}
\end{coro}
\begin{rem} \label{rem usual strichartz schrodinger sigma leq 2}
When $\sigma \in (0,2]\backslash\{1\}$, one can obtain the following global-in-time Strichartz estimate
\[
\|u\|_{L^p(\R,H^{\gamma-\gamma_{p,q}}_q)} \lesssim \|\varphi\|_{H^\gamma} + \|F\|_{L^1(\R, H^\gamma)}.
\]
It is valid for all $\gamma \in \R$.
\end{rem}
\noindent \textit{Proof of Corollary $\ref{coro local strichartz schrodinger}$.} We firstly note that when $\gamma_{p,q}\geq 0$ (or at least $\sigma \in (0,2] \backslash \{1\}$), we can obtain $(\ref{local strichartz schrodinger})$ for any $\gamma \in \R$ and $I = \R$. To see this, we write $\|u\|_{L^p(\R,H^{\gamma-\gamma_{p,q}}_q)}= \|\scal{D}^{\gamma-\gamma_{p,q}} u\|_{L^p(\R,L^q)}$ and use $(\ref{usual homogeneous strichartz schrodinger})$ with $\gamma=\gamma_{p,q}$ to obtain
\[
\|u\|_{L^p(\R,H^{\gamma-\gamma_{p,q}}_q)}\lesssim \|\scal{D}^{\gamma-\gamma_{p,q}} \varphi\|_{\dot{H}^{\gamma_{p,q}}} + \|\scal{D}^{\gamma-\gamma_{p,q}} F\|_{L^1(\R, \dot{H}^{\gamma_{p,q}})}.
\]
This gives the claim since $\|v\|_{\dot{H}^{\gamma_{p,q}}}\leq \|v\|_{H^{\gamma_{p,q}}}$ using that $\gamma_{p,q} \geq 0$. It remains to treat the case $\gamma_{p,q} <0$. By the Minkowski inequality and the unitary of $e^{-it\Lambda^\sigma}$ in $L^2$, the estimate $(\ref{local strichartz schrodinger})$ is proved if we can show for $\gamma \geq 0$, $I \subset \R$ a bounded interval and all $(p,q)$ admissible with $q<\infty$ that 
\begin{align}
\|e^{-it\Lambda^\sigma} \varphi \|_{L^p(I, H^{\gamma-\gamma_{p,q}}_q)} \lesssim \|\varphi\|_{H^\gamma}. \label{local strichartz schrodinger reduction}
\end{align}
Indeed, if we have $(\ref{local strichartz schrodinger reduction})$, then
\begin{align*}
\Big|\Big|\int_0^t e^{-i(t-s)\Lambda^\sigma} F(s)ds \Big|\Big|_{L^p(I, H^{\gamma-\gamma_{p,q}}_q)} &\leq \int_I \|\mathds{1}_{[0,t]}(s) e^{-i(t-s)\Lambda^\sigma} F(s)\|_{L^p(I, H^{\gamma-\gamma_{p,q}}_q)}ds  \\
&\leq \int_I \| e^{-i(t-s)\Lambda^\sigma} F(s) \|_{L^p(I, H^{\gamma-\gamma_{p,q}}_q)}ds \\
&\lesssim \int_I \|F(s)\|_{H^\gamma} ds=\|F\|_{L^1(I,H^\gamma)}.
\end{align*}
We now prove $(\ref{local strichartz schrodinger reduction})$. To do so, we write
\[
\scal{D}^{\gamma-\gamma_{p,q}}e^{-it\Lambda^\sigma} \varphi = \psi(D) \scal{D}^{\gamma-\gamma_{p,q}} e^{-it\Lambda^\sigma} \varphi + (1-\psi)(D)\scal{D}^{\gamma-\gamma_{p,q}} e^{-it\Lambda^\sigma} \varphi,
\]
for some $\psi \in C^\infty_0(\R^d)$ valued in $[0,1]$ and equal to 1 near the origin. For the first term, the Sobolev embedding implies
\[
\|\psi(D) \scal{D}^{\gamma-\gamma_{p,q}} e^{-it\Lambda^\sigma} \varphi\|_{L^q} \lesssim \|\psi(D) \scal{D}^{\gamma-\gamma_{p,q}} e^{-it\Lambda^\sigma} \varphi\|_{H^\delta},
\]
for some $\delta >d/2-d/q$. Thanks to the support of $\psi$ and the unitary property of $e^{-it\Lambda^\sigma}$ in $L^2$, we get
\[
\|\psi(D) \scal{D}^{\gamma-\gamma_{p,q}} e^{-it\Lambda^\sigma} \varphi\|_{L^p(I, L^q)} \lesssim \|\varphi\|_{L^2} \lesssim \|\varphi\|_{H^\gamma}.
\]
Here the boundedness of $I$ is crucial to have the first estimate. For the second term, using $(\ref{usual inhomogeneous strichartz schrodinger})$, we obtain
\[
\|(1-\psi)(D) \scal{D}^{\gamma-\gamma_{p,q}} e^{-it\Lambda^\sigma} \varphi\|_{L^p(I, L^q)} \lesssim \|(1-\psi)(D)\scal{D}^{\gamma-\gamma_{p,q}} \varphi\|_{\dot{H}^{\gamma_{p,q}}} \lesssim \|\varphi\|_{H^\gamma}.
\]
Combining the two terms, we have $(\ref{local strichartz schrodinger reduction})$. This completes the proof. 
\defendproof
\begin{coro} \label{coro full strichartz wave}
Let $d\geq 1, \sigma \in (0,\infty)\backslash \{1\}, \gamma \in \R$ and a (weak) solution to the linear fractional wave equation, namely
\[
v(t)= \cos (t\Lambda^\sigma) \varphi + \frac{\sin (t\Lambda^\sigma)}{\Lambda^\sigma} \phi + \int_0^t \frac{\sin((t-s)\Lambda^\sigma)}{\Lambda^\sigma}G(s)ds, 
\]
for some data $\varphi, \phi, G$. Then for all $(p,q)$ and $(a,b)$ admissible pairs,
\begin{align}
\|[v]\|_{L^p(\R,\dot{B}^\gamma_{q})} \lesssim \|[v](0)\|_{\dot{H}^{\gamma+\gamma_{p,q}}} + \|G\|_{L^{a'}(\R,\dot{B}^{\gamma+\gamma_{p,q}-\gamma_{a',b'}-2\sigma}_{b'})}, \label{full strichartz wave}
\end{align}
where 
\begin{align*}
\|[v]\|_{L^p(\R,\dot{B}^\gamma_{q})} &:=\|v\|_{L^p(\R,\dot{B}^\gamma_{q})} + \|\partial_t v\|_{L^p(\R,\dot{B}^{\gamma-\sigma}_{q})}; \\
\|[v](0)\|_{\dot{H}^{\gamma+\gamma_{p,q}}}&:=\|\varphi\|_{\dot{H}^{\gamma+\gamma_{p,q}}} +\|\phi\|_{\dot{H}^{\gamma+\gamma_{p,q}-\sigma}}.
\end{align*}
In particular, 
\begin{align}
\|[v]\|_{L^p(\R,\dot{B}^{\gamma-\gamma_{p,q}}_{q})} \lesssim \|[v](0)\|_{\dot{H}^\gamma} + \|G\|_{L^1(\R,\dot{H}^{\gamma-\sigma})}, \label{homogeneous strichartz wave}
\end{align}
and 
\begin{align}
\|[v]\|_{L^\infty(\R,\dot{B}^{\gamma_{p,q}}_{2})}+\|[v]\|_{L^p(\R,\dot{B}^0_{q})} \lesssim \|[v](0)\|_{\dot{H}^{\gamma_{p,q}}} + \|G\|_{L^{a'}(\R,\dot{B}^0_{b'})}, \label{inhomogeneous strichartz wave}
\end{align}
provided that 
\begin{align}
\gamma_{p,q}= \gamma_{a',b'}+2\sigma. \label{gap condition wave}
\end{align}
\end{coro}
\begin{proof}
It follows easily from Theorem $\ref{theorem full strichartz schrodinger}$ and the fact that
\[
\cos (t\Lambda^\sigma) = \frac{e^{it\Lambda^\sigma}+e^{-it\Lambda^\sigma}}{2}, \quad \sin (t\Lambda^\sigma) = \frac{e^{it\Lambda^\sigma}-e^{-it\Lambda^\sigma}}{2i}.
\]
\end{proof}
As in Corollary $\ref{coro usual strichartz schrodinger}$, we have the following usual Strichartz estimates for the fractional wave equation.
\begin{coro} \label{coro usual strichartz wave}
Let $d\geq 1, \sigma \in (0,\infty)\backslash \{1\}, \gamma \in \R$. If $v$ is a (weak) solution to the linear fractional wave equation for some data $\varphi, \phi, G$, then for all $(p,q)$ and $(a,b)$ admissible satisfying $q<\infty, b<\infty$ and $(\ref{gap condition wave})$, 
\begin{align}
\|v\|_{L^p(\R,\dot{H}^{\gamma-\gamma_{p,q}}_q)} &\lesssim \|[v](0)\|_{\dot{H}^\gamma} + \|G\|_{L^1(\R,\dot{H}^{\gamma-\sigma})}, \label{usual homogeneous strichartz wave} \\
\|[v]\|_{L^\infty(\R,\dot{H}^{\gamma_{p,q}})}+ \|v\|_{L^p(\R, L^q)} &\lesssim \|[v](0)\|_{\dot{H}^{\gamma_{p,q}}} + \|G\|_{L^{a'}(\R,L^{b'})}. \label{usual inhomogeneous strichartz wave}
\end{align}
\end{coro}
The following result, which is similar to Corollary $\ref{coro local strichartz schrodinger}$, gives the local Strichartz estimates for the fractional wave equation. 
\begin{coro} \label{coro local strichartz wave}
Let $d \geq 1$, $\sigma \in (0,\infty)\backslash \{1\}$, $\gamma \geq 0$ and $I \subset \R$ a bounded interval. If $v$ is a (weak) solution to the linear fractional wave equation for some data $\varphi, \phi, G$, then for all $(p,q)$ admissible satisfying $q <\infty$,
\begin{align}
\|v\|_{L^p(I,H^{\gamma-\gamma_{p,q}}_q)} \lesssim \|[v](0)\|_{H^\gamma} + \|G\|_{L^1(I, H^{\gamma-\sigma})}. \label{local strichartz wave}
\end{align}
\end{coro}
\begin{proof}
The proof is similar to the one of Corollary $\ref{coro local strichartz schrodinger}$. Thanks to the Minkowski inequality, it suffices to prove for all $\gamma \geq 0$, all $I \subset \R$ bounded interval and all $(p,q)$ admissible pair with $q<\infty$,
\begin{align}
\|\cos (t\Lambda^\sigma) \varphi \|_{L^p(I, H^{\gamma-\gamma_{p,q}}_q)} & \lesssim  \|\varphi\|_{H^\gamma}, \label{reduction local strichartz fractional wave 1} \\
\Big|\Big| \frac{\sin (t\Lambda^\sigma)}{\Lambda^\sigma} \phi \Big|\Big|_{L^p(I,H^{\gamma-\gamma_{p,q}}_q)} & \lesssim  \|\phi\|_{H^{\gamma-\sigma}}. \label{reduction local strichartz fractional wave 2}
\end{align}
The estimate $(\ref{reduction local strichartz fractional wave 1})$ follows from the ones of $e^{\pm it\Lambda^\sigma}$. We will give the proof of $(\ref{reduction local strichartz fractional wave 2})$. To do this, we write
\[
\scal{D}^{\gamma-\gamma_{p,q}}\frac{\sin (t\Lambda^\sigma)}{\Lambda^\sigma} = \psi(D) \scal{D}^{\gamma-\gamma_{p,q}}\frac{\sin (t\Lambda^\sigma)}{\Lambda^\sigma} + (1-\psi)(D)\scal{D}^{\gamma-\gamma_{p,q}}\frac{\sin (t\Lambda^\sigma)}{\Lambda^\sigma},
\]
for some $\psi$ as in the proof of Corollary $\ref{coro local strichartz schrodinger}$. For the first term, the Sobolev embedding and the fact $\Big|\Big|\frac{\sin (t\Lambda^\sigma)}{\Lambda^\sigma} \Big|\Big|_{L^2\rightarrow L^2} \leq |t|$ imply
\[
\Big|\Big| \psi(D) \scal{D}^{\gamma-\gamma_{p,q}} \frac{\sin(t\Lambda^\sigma)}{\Lambda^\sigma} \phi\Big| \Big|_{L^q} \lesssim |t| \|\psi(D) \scal{D}^{\gamma+\delta-\gamma_{p,q}} \phi\|_{L^2},
\]
for some $\delta>d/2-d/q$. This gives
\[
\Big|\Big| \psi(D) \scal{D}^{\gamma-\gamma_{p,q}} \frac{\sin(t\Lambda^\sigma)}{\Lambda^\sigma} \phi\Big| \Big|_{L^p(I,L^q)} \lesssim \|\phi\|_{H^{\gamma-\sigma}}.
\]
Here we use the fact that $\|\psi(D) \scal{D}^{\delta+\sigma-\gamma_{p,q}}\|_{L^2\rightarrow L^2} \lesssim 1$. For the second term, we apply $(\ref{local strichartz schrodinger reduction})$ with the fact $\sin (t\Lambda^\sigma)=(e^{it\Lambda^\sigma}-e^{-it\Lambda^\sigma})/2i$ and get
\[
\Big|\Big| (1-\psi)(D) \scal{D}^{\gamma-\gamma_{p,q}} \frac{\sin(t\Lambda^\sigma)}{\Lambda^\sigma} \phi\Big| \Big|_{L^p(I,L^q)} \lesssim \|(1-\psi)(D) \Lambda^{-\sigma} \phi \|_{H^\gamma} \lesssim \|\phi\|_{H^{\gamma-\sigma}}.
\]
Here we use that $\|(1-\psi)(D) \scal{D}^{\sigma}\Lambda^{-\sigma}\|_{L^2\rightarrow L^2} \lesssim 1$ by functional calculus. Combining two terms, we have $(\ref{reduction local strichartz fractional wave 2})$. The proof is complete. 
\end{proof}
\section{Nonlinear estimates} \label{section nonlinear estimates}
\setcounter{equation}{0}
In this section, we recall some estimates related to the fractional derivatives of nonlinear operators. Let us start with the following Kato-Ponce inequality (or fractional Leibniz rule).
\begin{prop} \label{prop fractional leibniz rule}
Let $\gamma \geq 0, 1 < r < \infty$ and $1 < p_1, p_2, q_1, q_2 \leq \infty$ satisfying $\frac{1}{r}=\frac{1}{p_1}+\frac{1}{q_1}=\frac{1}{p_2}+\frac{1}{q_2}$. Then there exists $C=C(d,\gamma,r,p_1,q_1,p_2,q_2)>0$ such that for all $u,v \in \Sh$, 
\begin{align}
\|\Lambda^\gamma(uv)\|_{L^r} &\leq C \Big(\|\Lambda^\gamma u\|_{L^{p_1}} \|v\|_{L^{q_1}} + \|u\|_{L^{p_2}}\|\Lambda^\gamma v\|_{L^{q_2}} \Big), \label{leibniz rule homogeneous sobolev} \\
\|\scal{\Lambda}^\gamma(uv)\|_{L^r} &\leq C \Big(\|\scal{\Lambda}^\gamma u\|_{L^{p_1}} \|v\|_{L^{q_1}} + \|u\|_{L^{p_2}}\|\scal{\Lambda}^\gamma v\|_{L^{q_2}} \Big). \label{leibniz rule inhomogeneous sobolev} 
\end{align}
\end{prop}
We refer to $\cite{GrafakosOh}$ (and references therein) for the proof of above inequalities and more general results. We also have the following fractional chain rule.
\begin{prop} \label{prop fractional chain rule}
Let $F \in C^1(\C, \C)$ and $G \in C(\C, \R^+)$ such that $F(0)=0$ and 
\[
|F'( \theta z+ (1-\theta) \zeta)| \leq \mu(\theta) (G(z)+G(\zeta)), \quad z,\zeta \in \C, \quad 0 \leq \theta \leq 1,
\]
where $\mu \in L^1((0,1))$. Then for $ \gamma \in (0,1)$ and $1 <r, p <\infty$, $1 <q \leq \infty$ satisfying $\frac{1}{r}=\frac{1}{p}+\frac{1}{q}$, there exists $C=C(d,\mu, \gamma, r, p,q)>0$ such that for all $u \in \Sh$,
\begin{align}
\|\Lambda^\gamma F(u) \|_{L^r} &\leq C \|F'(u) \|_{L^q} \|\Lambda^\gamma u \|_{L^p},\label{chain rule homogeneous sobolev} \\
\|\scal{\Lambda}^\gamma F(u) \|_{L^r} &\leq C \|F'(u) \|_{L^q} \|\scal{\Lambda}^\gamma u \|_{L^p}. \label{chain rule inhomogeneous sobolev} 
\end{align}
\end{prop}
We refer to \cite{ChristWeinstein} (see also \cite{Staffilani}) for the proof of $(\ref{chain rule homogeneous sobolev})$ and \cite{Taylor} for $(\ref{chain rule inhomogeneous sobolev})$. A direct consequence of the fractional Leibniz rule and the fractional chain rule is the following fractional derivatives.
\begin{lem} \label{lem nonlinear estimates}
Let $F \in C^k(\C, \C), k \in \N \backslash \{0\}$. Assume that there is $\nu \geq k$  such that 
\[
|D^iF (z)| \leq C |z|^{\nu-i}, \quad z \in \C, \quad i=1,2,...., k.
\]
Then for $\gamma \in [0,k]$ and $1 <r, p <\infty$, $1 <q \leq \infty$ satisfying $\frac{1}{r}=\frac{1}{p}+\frac{\nu-1}{q}$, there exists $C=C(d, \nu, \gamma, r,p,q)>0$ such that for all $u \in \Sh$,
\begin{align}
\|\Lambda^\gamma F(u)\|_{L^r} &\leq C \|u\|^{\nu-1}_{L^q} \|\Lambda^\gamma u \|_{L^p}, \label{nonlinear estimate homogeneous sobolev} \\
\|\scal{\Lambda}^\gamma F(u) \|_{L^r} & \leq C\|u\|^{\nu-1}_{L^q} \|\scal{\Lambda}^\gamma u\|_{L^p}. \label{nonlinear estimate inhomogeneous sobolev}
\end{align}
\end{lem}
The reader can find the proof of $(\ref{nonlinear estimate homogeneous sobolev})$ in \cite{Kato95}. The one of $(\ref{nonlinear estimate inhomogeneous sobolev})$ follows from $(\ref{nonlinear estimate homogeneous sobolev})$, the H\"older inequality and the fact that
\[
\|\scal{\Lambda}^\gamma u \|_{L^r} \sim \|u\|_{L^r} + \|\Lambda^\gamma u \|_{L^r},
\]
for $1<r<\infty, \gamma>0$.
\begin{lem} \label{lem nonlinear estimates polynomial}
Let $F(z)$ be a homogeneous polynomial in $z,\overline{z}$ of degree $\nu\geq 1$. Then $(\ref{nonlinear estimate homogeneous sobolev})$ and $(\ref{nonlinear estimate inhomogeneous sobolev})$ hold true for any $\gamma \geq 0$ and $r, p, q$ as in $\emph{Lemma}$ $\ref{lem nonlinear estimates}$.
\end{lem}
\begin{coro} \label{coro fractional derivatives}
Let $F(z)=|z|^{\nu-1}z$ with $\nu>1$, $\gamma \geq 0$ and $1 <r, p <\infty$, $1 <q \leq \infty$ satisfying $\frac{1}{r}=\frac{1}{p}+\frac{\nu-1}{q}$. 
\begin{itemize}
\item[\emph{i.}] If $\nu$ is an odd integer or $\lceil \gamma \rceil \leq \nu$ otherwise, then there exists $C=C(d,\nu, \gamma, r, p, q)>0$ such that for all $u \in \Sh$,
\[
\|F(u)\|_{\dot{H}^\gamma_r}  \leq C \|u\|^{\nu-1}_{L^q} \|u\|_{\dot{H}^\gamma_p}.
\]
A similar estimate holds with $\dot{H}^\gamma_r, \dot{H}^\gamma_p$-norms are replaced by $H^\gamma_r, H^\gamma_p$-norms respectively.
\item[\emph{ii.}] If $\nu$ is an odd integer or $\lceil \gamma \rceil \leq \nu-1$ otherwise, then there exists $C= C(d,\nu,\gamma, r, p, q)>0$ such that for all $u, v \in \Sh$,
\begin{multline}
\|F(u)-F(v)\|_{\dot{H}^\gamma_r}  \leq C \Big( (\|u\|^{\nu-1}_{L^q} + \|v\|^{\nu-1}_{L^q}) \|u-v\|_{\dot{H}^\gamma_p} \Big. \\
\Big.+ (\|u\|^{\nu-2}_{L^q} +\|v\|^{\nu-2}_{L^q})(\|u\|_{\dot{H}^\gamma_p} + \|v\|_{\dot{H}^\gamma_p}) \|u-v\|_{L^q} \Big). \nonumber
\end{multline}
A similar estimate holds with $\dot{H}^\gamma_r, \dot{H}^\gamma_p$-norms are replaced by $H^\gamma_r, H^\gamma_p$-norms respectively.
\end{itemize}
\begin{proof}
Item 1 is an immediate consequence of Lemma $\ref{lem nonlinear estimates}$ and Lemma $\ref{lem nonlinear estimates polynomial}$. For Item 2, we firstly write
\[
F(u)-F(v)= \nu \int_{0}^{1} |v+t(u-v)|^{\nu-1} (u-v) dt,
\]
and use the fractional Leibniz rule given in Proposition $\ref{prop fractional leibniz rule}$. Then the results follows by applying the fractional derivatives given in Lemma $\ref{lem nonlinear estimates}$ and Lemma $\ref{lem nonlinear estimates polynomial}$. 
\end{proof}
\end{coro}
\section{Nonlinear fractional Schr\"odinger equations}
\setcounter{equation}{0}
\subsection{Local well-posedness in sub-critical cases}
In this subsection, we give the proofs of Theorem $\ref{theorem local wellposedness subcritical schrodinger sigma <2}$, Theorem $\ref{theorem local wellposedness subcritical schrodinger sigma geq 2}$ and Proposition $\ref{prop global subcritical schrodinger}$. 
\paragraph{Proof of Theorem $\ref{theorem local wellposedness subcritical schrodinger sigma <2}$.} 
We follow the standard process (see e.g. \cite{Cazenave} or \cite{BGT}) by using the fixed point argument in a suitable Banach space. We firstly choose $p> \max(\nu-1, 4)$ when $d=1$ and $p> \max(\nu-1, 2)$ when $d\geq 2$ such that $\gamma > d/2 - \sigma/p$ and then choose $q\in [2,\infty)$ such that
\[
\frac{2}{p} + \frac{d}{q} \leq  \frac{d}{2}.
\]
\textbf{Step 1.} Existence. Let us consider
\[
X := \Big\{ u \in L^\infty(I,H^\gamma) \cap L^p(I,H^{\gamma - \gamma_{p,q}}_q) \ | \ \|u\|_{L^\infty(I, H^\gamma)} + \|u\|_{L^p(I, H^{\gamma-\gamma_{p,q}}_q)} \leq M \Big\},
\]
equipped with the distance
\[
d(u,v):= \|u-v\|_{L^\infty(I, L^2)} + \|u-v\|_{L^p(I,H^{-\gamma_{p,q}}_q)},
\]
where $I =[0,T]$ and $M, T>0$ to be chosen later. The persistence of regularity (see e.g. \cite{Cazenave}, Theorem 1.2.5) shows that $(X,d)$ is a complete metric space. By the Duhamel formula, it suffices to prove that the functional 
\begin{align}
\Phi(u)(t) = e^{-it\Lambda^\sigma} \varphi +i\mu \int_0^t e^{-i(t-s)\Lambda^\sigma} |u(s)|^{\nu-1} u(s)ds \label{duhamel formula schrodinger}
\end{align}
is a contraction on $X$. The local Strichartz estimate $(\ref{local strichartz schrodinger})$ gives
\begin{align*}
\|\Phi(u)\|_{L^\infty(I,H^\gamma)}+\|\Phi(u)\|_{L^p(I, H^{\gamma-\gamma_{p,q}}_q)} &\lesssim \|\varphi\|_{H^\gamma}+ \|F(u)\|_{L^1(I,H^\gamma)}, \\
\|\Phi(u)-\Phi(v)\|_{L^\infty(I,L^2)} +\|\Phi(u)-\Phi(v)\|_{L^p(I,H^{-\gamma_{p,q}}_q)} &\lesssim \|F(u)-F(v)\|_{L^1(I,L^2)},
\end{align*}
where $F(u) = |u|^{\nu-1} u$. By our assumptions on $\nu$, Corollary $\ref{coro fractional derivatives}$ gives
\begin{align}
\|F(u)\|_{L^1(I,H^\gamma)} &\lesssim \|u\|^{\nu-1}_{L^{\nu-1}(I,L^\infty)} \|u\|_{L^\infty(I,H^\gamma)} \lesssim  T^{1-\frac{\nu-1}{p}} \|u\|^{\nu-1}_{L^p(I,L^\infty)} \|u\|_{L^\infty(I,H^\gamma)}, \label{nonlinear application 1} \\
\|F(u)-F(v)\|_{L^1(I,L^2)} &\lesssim \Big(\|u\|^{\nu-1}_{L^{\nu-1}(I,L^\infty)} + \|v\|^{\nu-1}_{L^{\nu-1}(I,L^\infty)}\Big) \|u-v\|_{L^\infty(I,L^2)}  \nonumber \\
&\lesssim T^{1-\frac{\nu-1}{p}} \Big(\|u\|^{\nu-1}_{L^p(I,L^\infty)} + \|v\|^{\nu-1}_{L^p(I,L^\infty)}\Big) \|u-v\|_{L^\infty(I,L^2)}. \label{uniqueness subcritical schrodinger}
\end{align}
Using that $\gamma - \gamma_{p,q} > d/q$, the Sobolev embedding implies $L^p(I,H^{\gamma - \gamma_{p,q}}_q) \subset L^p(I,L^\infty)$. Thus, we get
\[
\|\Phi(u)\|_{L^\infty(I,H^\gamma)}+\|\Phi(u)\|_{L^p(I, H^{\gamma-\gamma_{p,q}}_q)} \lesssim \|\varphi\|_{H^\gamma}+ T^{1-\frac{\nu-1}{p}} \|u\|^{\nu-1}_{L^p(I,H^{\gamma-\gamma_{p,q}}_q)} \|u\|_{L^\infty(I,H^\gamma)}, 
\]
and
\[
d(\Phi(u),\Phi(v))  \lesssim T^{1-\frac{\nu-1}{p}} \Big(\|u\|^{\nu-1}_{L^p(I,H^{\gamma-\gamma_{p,q}}_q)} + \|v\|^{\nu-1}_{L^p(I,H^{\gamma-\gamma_{p,q}}_q)}\Big) \|u-v\|_{L^\infty(I,L^2)}. 
\]
This shows that for all $u,v \in X$, there exists $C>0$ independent of $\varphi \in H^\gamma$ such that
\begin{align*}
\|\Phi(u)\|_{L^\infty(I,H^\gamma)}+\|\Phi(u)\|_{L^p(I, H^{\gamma-\gamma_{p,q}}_q)} &\leq C\|\varphi\|_{H^\gamma} + C T^{1-\frac{\nu-1}{p}} M^{\nu}, \\
d(\Phi(u),\Phi(v))  &\leq C T^{1-\frac{\nu-1}{p}} M^{\nu-1} d(u,v).
\end{align*}
Therefore, if we set $M=2C\|\varphi\|_{H^\gamma}$ and choose $T>0$ small enough so that $C T^{1-\frac{\nu-1}{p}} M^{\nu-1} \leq \frac{1}{2}$, then $X$ is stable by $\Phi$ and $\Phi$ is a contraction on $X$. By the fixed point theorem, there exists a unique $u \in X$ so that $\Phi(u)=u$. \newline
\noindent \textbf{Step 2.} Uniqueness. Consider $u, v \in C(I,H^\gamma) \cap L^p(I, L^\infty)$ two solutions of (NLFS). Since the uniqueness is a local property (see \cite{Cazenave}), it suffices to show $u=v$ for $T$ is small. We have from $(\ref{uniqueness subcritical schrodinger})$ that
\[
d(u,v) \leq C T^{1-\frac{\nu-1}{p}} \Big( \|u\|^{\nu-1}_{L^p(I,L^\infty)} + \|v\|^{\nu-1}_{L^p(I,L^\infty)}\Big) d(u,v).
\] 
Since $\|u\|_{L^p(I, L^\infty)}$ is small if $T$ is small and similarly for $v$, we see that if $T>0$ small enough, 
\[
d(u,v) \leq \frac{1}{2} d(u,v) \text{ or } u=v.
\]
\noindent \textbf{Step 3.} Item i. Since the time of existence constructed in Step 1 only depends on $\|\varphi\|_{H^\gamma}$. The blowup alternative follows by standard argument (see e.g. \cite{Cazenave}). \newline
\noindent \textbf{Step 4.} Item ii. Let $\varphi_n \rightarrow \varphi$ in $H^\gamma$ and $C, T=T(\varphi)$ be as in Step 1. Set $M=4C\|\varphi\|_{H^\gamma}$. It follows that $2C\|\varphi_n\|_{H^\gamma} \leq M$ for sufficiently large $n$. Thus the solution $u_n$ constructed in Step 1 belongs to $X$ with $T=T(\varphi)$ for $n$ large enough. We have from Strichartz estimate $(\ref{local strichartz schrodinger})$ and $(\ref{nonlinear application 1})$ that
\[
\|u\|_{L^a(I,H^{\gamma-\gamma_{a,b}}_b)} \lesssim \|\varphi\|_{H^\gamma}+T^{1-\frac{\nu-1}{p}} \|u\|^{\nu-1}_{L^p(I,L^\infty)} \|u\|_{L^\infty(I,H^\gamma)},
\]
provided $(a,b)$ is admissible and $b<\infty$. This shows the boundedness of $u_n$ in $L^a(I,H^{\gamma-\gamma_{a,b}}_b)$. We also have from $(\ref{uniqueness subcritical schrodinger})$ and the choice of $T$ that
\[
d(u_n, u) \leq C\|\varphi_n-\varphi\|_{L^2}+\frac{1}{2}d(u_n,u) \text{ or } d(u_n, u) \leq 2C\|\varphi_n-\varphi\|_{L^2}.
\]
This yields that $u_n \rightarrow u$ in $L^\infty(I,L^2)\cap L^p(I,H^{-\gamma_{p,q}}_q)$. Strichartz estimate $(\ref{local strichartz schrodinger})$ again implies that $u_n\rightarrow u$ in $L^a(I,H^{-\gamma_{a,b}}_b)$ for any admissible pair $(a,b)$ with $b<\infty$. The convergence in $C(I,H^{\gamma-\ep})$ follows from the boundedness in $L^\infty(I,H^\gamma)$, the convergence in $L^\infty(I,L^2)$ and that $\|u\|_{H^{\gamma-\ep}} \leq \|u\|^{1-\frac{\ep}{\gamma}}_{H^\gamma}\|u\|^{\frac{\ep}{\gamma}}_{L^2}$.
\defendproof
\begin{rem} \label{rem continuity in H gamma}
If we assume that $\nu>1$ is an odd integer or 
\[
\lceil \gamma \rceil \leq \nu-1
\]
otherwise, then the continuous dependence holds in $C(I,H^\gamma)$. To see this, we consider $X$ as above equipped with the following metric
\[
d(u,v):= \|u-v\|_{L^\infty(I,H^\gamma)} + \|u-v\|_{L^p(I,H^{\gamma-\gamma_{p,q}}_q)}.
\]
Using Item ii of Corollary $\ref{coro fractional derivatives}$, we have
\begin{multline*}
\|F(u)-F(v)\|_{L^1(I,H^\gamma)} \lesssim  (\|u\|^{\nu-1}_{L^{\nu-1}(I,L^\infty)} +\|v\|^{\nu-1}_{L^{\nu-1}(I,L^\infty)}) \|u-v\|_{L^\infty(I,H^\gamma)} \\ + (\|u\|^{\nu-2}_{L^{\nu-1}(I,L^\infty)}+\|v\|^{\nu-2}_{L^{\nu-1}(I,L^\infty)})(\|u\|_{L^\infty(I,H^\gamma)} + \|v\|_{L^\infty(I,H^\gamma)})\|u-v\|_{L^{\nu-1}(I,L^\infty)}.
\end{multline*}
Using the Sobolev embedding, we see that for all $u, v \in X$,
\[
d(\Phi(u),\Phi(v)) \lesssim T^{1-\frac{\nu-1}{p}} M^{\nu-1} d(u,v).
\]
Therefore, the continuity in $C(I,H^\gamma)$ follows as in Step 5.
\end{rem}
\paragraph{Proof of Theorem $\ref{theorem local wellposedness subcritical schrodinger sigma geq 2}$.} 
Let $(p,q)$ be as in $(\ref{define p q})$. It is easy to see that $(p,q)$ is admissible and $\gamma_{p,q}=0=\gamma_{p',q'}+\sigma$. We next choose $(m, n)$ so that
\begin{align}
\frac{1}{p'}=\frac{1}{p}+\frac{\nu-1}{m}, \quad \frac{1}{q'}=\frac{1}{q}+\frac{\nu-1}{n}. \label{define m n}
\end{align}
It is easy to see that
\[
\frac{\nu-1}{m}-\frac{\nu-1}{p}= 1-\frac{(\nu-1)(d-2\gamma)}{2\sigma} >0, \quad q \leq n = \frac{dq}{d-\gamma q}.
\]
The Sobolev embedding implies
\begin{align}
\|u\|^{\nu-1}_{L^m(I,L^n)} \lesssim |I|^{1-\frac{(\nu-1)(d-2\gamma)}{2\sigma}} \|u\|^{\nu-1}_{L^p(I,\dot{H}^\gamma_q)}. \label{sobolev embedding}
\end{align}
\textbf{Step 1.} Existence. Let us consider
\[
X:=\Big\{u \in L^p(I,H^\gamma_q) \ | \ \|u\|_{L^p(I,\dot{H}^\gamma_q)} \leq M \Big\},
\]
equipped with the distance
\[
d(u,v)=\|u-v\|_{L^p(I,L^q)},
\]
where $I=[0,T]$ and $M,T>0$ to be determined. One can easily verify that $(X,d)$ is a complete metric space (see e.g. \cite{CazenaveWeissler}). The Strichartz estimate $(\ref{usual inhomogeneous strichartz schrodinger})$ implies
\begin{align*}
\|\Phi(u)\|_{L^p(I,\dot{H}^\gamma_q)} &\lesssim \|\varphi\|_{\dot{H}^\gamma} + \|F(u)\|_{L^{p'}(I,\dot{H}^\gamma_{q'})}, \\
\|\Phi(u)-\Phi(v)\|_{L^p(I,L^q)} &\lesssim \|F(u)-F(v)\|_{L^{p'}(I,L^{q'})}.
\end{align*}
It follows from Corollary $\ref{coro fractional derivatives}$, $(\ref{define m n})$ and $(\ref{sobolev embedding})$ that
\begin{align}
\|\Phi(u)\|_{L^p(I,\dot{H}^\gamma_q)} & \lesssim \|\varphi\|_{\dot{H}^\gamma}+ T^{1-\frac{(\nu-1)(d-2\gamma)}{2\sigma}} \|u\|^\nu_{L^p(I,\dot{H}^\gamma_q)}, \label{item i subcritical sigma geq 2}\\
\|\Phi(u)-\Phi(v)\|_{L^p(I,L^q)} &\lesssim T^{1-\frac{(\nu-1)(d-2\gamma)}{2\sigma}} \Big(\|u\|^{\nu-1}_{L^p(I,\dot{H}^\gamma_q)} +\|v\|^{\nu-1}_{L^p(I,\dot{H}^\gamma_q)} \Big) \|u-v\|_{L^p(I,L^q)}. \label{uniqueness subcritical sigma geq 2}
\end{align}
This implies for all $u, v \in X$, there exists $C$ independent of $\varphi \in H^\gamma$ such that
\begin{align}
\|\Phi(u)\|_{L^p(I,\dot{H}^\gamma_q)} &\leq C \|\varphi\|_{\dot{H}^\gamma} + C T^{1-\frac{(\nu-1)(d-2\gamma)}{2\sigma}} M^\nu, \nonumber \\
d(\Phi(u),\Phi(v)) &\leq C T^{1-\frac{(\nu-1)(d-2\gamma)}{2\sigma}} M^{\nu-1}d(u,v). \nonumber
\end{align}
If we set $M=2C\|\varphi\|_{\dot{H}^\gamma}$ and choose $T>0$ small enough so that $C T^{1-\frac{(\nu-1)(d-2\gamma)}{2\sigma}} M^{\nu-1} \leq \frac{1}{2}$, then $\Phi$ is a strict contraction on $X$. Thus $\Phi$ has a unique fixed point in $X$. Since $\varphi \in H^\gamma$ and $u \in L^p(I,H^\gamma_q)$, the continuity in $H^\gamma$ follows easily from Strichartz estimates (see e.g. \cite{CazenaveWeissler}). This proves the existence of solution $u \in C(I,H^\gamma)\cap L^p(I,H^\gamma_q)$ to (NLFS).\newline
\textbf{Step 2.} Uniqueness. The uniqueness is similar to Step 2 of the proof of Theorem $\ref{theorem local wellposedness subcritical schrodinger sigma <2}$ using $(\ref{uniqueness subcritical sigma geq 2})$. Note that $\|u\|_{L^p(I,\dot{H}^\gamma_q)}$ can be small if $T$ is taken small enough.\newline
\textbf{Step 3.} Item i. The blowup alternative is easy since the time of existence depends only on $\|\varphi\|_{\dot{H}^\gamma}$. \newline
\textbf{Step 4.} Item ii. The continuous dependence is similar to that of Theorem $\ref{theorem local wellposedness subcritical schrodinger sigma <2}$. We have from Strichartz estimate $(\ref{usual inhomogeneous strichartz schrodinger})$ and $(\ref{item i subcritical sigma geq 2})$ that
\begin{align*}
\|u\|_{L^a(I,\dot{H}^\gamma_b)} &\lesssim \|\varphi\|_{\dot{H}^\gamma}+T^{1-\frac{(\nu-1)(d-2\gamma)}{2\sigma}} \|u\|^\nu_{L^p(I,\dot{H}^\gamma_q)}, \\
\|u\|_{L^a(I,L^b)} &\lesssim \|\varphi\|_{L^2} + T^{1-\frac{(\nu-1)(d-2\gamma)}{2\sigma}} \|u\|^{\nu-1}_{L^p(I,\dot{H}^\gamma_q)} \|u\|_{L^p(I,L^q)},
\end{align*}
provided that $(a,b)$ is admissible, $b<\infty$ and $\gamma_{a,b}=0$. This gives the boundedness of $u_n$ in $L^a(I,H^\gamma_b)$. The convergence in $L^a(I, L^b)$ and $H^{\gamma-\ep}$ follows similarly as in Step 4 of Theorem $\ref{theorem local wellposedness subcritical schrodinger sigma <2}$ using $(\ref{uniqueness subcritical sigma geq 2})$.
\defendproof
\paragraph{Proof of Proposition $\ref{prop global subcritical schrodinger}$.}
The assumption $(\ref{assumption global existence})$ allows us to apply Theorem $\ref{theorem local wellposedness subcritical schrodinger sigma <2}$ and Theorem $\ref{theorem local wellposedness subcritical schrodinger sigma geq 2}$ with $\gamma=\sigma/2$ and obtain the local well-posedness in $H^{\sigma/2}$. We now prove the global extension using the blowup alternative. Item i follows from the conservation of mass and energy. For Item ii and Item iii, we firstly use Gagliardo-Nirenberg's inequality (see e.g. \cite{Tao}, Appendix) with the fact that
\[
\frac{1}{\nu+1}=\frac{1}{2}-\frac{\theta \sigma}{2d} \text{ or } \theta =\frac{d(\nu-1)}{\sigma(\nu+1)}
\]
and the conservation of mass to get
\[
\|u(t)\|^{\nu+1}_{L^{\nu+1}} \lesssim \|\Lambda^{\sigma/2} u(t)\|^{\frac{d(\nu-1)}{\sigma}}_{L^2} \|u(t)\|^{\nu+1-\frac{d(\nu-1)}{\sigma}}_{L^2} = \|u(t)\|^{\frac{d(\nu-1)}{\sigma}}_{\dot{H}^{\sigma/2}}\|\varphi\|^{\nu+1-\frac{d(\nu-1)}{\sigma}}_{L^2}.
\]  
Note that here the assumption $\nu\leq 1+2\sigma/d$ ensures that $\theta \in (0,1)$. The conservation of mass then gives
\[
\frac{1}{2}\|u(t)\|^2_{\dot{H}^{\sigma/2}} = E_{\text{s}}(u(t)) -\frac{\mu}{\nu+1}\|u(t)\|^{\nu+1}_{L^{\nu+1}} \lesssim E_{\text{s}}(\varphi) -\frac{\mu}{\nu+1} \|u(t)\|^{\frac{d(\nu-1)}{\sigma}}_{\dot{H}^{\sigma/2}}\|\varphi\|^{\nu+1-\frac{d(\nu-1)}{\sigma}}_{L^2}.
\]
If $\nu\in (1,1+2\sigma/d)$ or $\frac{d(\nu-1)}{\sigma} \in (0,2)$, then $\|u(t)\|_{\dot{H}^{\sigma/2}} \leq C$. This together with the conservation of mass implies the boundedness of $\|u(t)\|_{H^{\sigma/2}}$ and Item ii follows. Item iii is treated similarly with $\|\varphi\|_{L^2}$ is small. It remains to show Item iv. By Sobolev embedding with $\frac{1}{2}\leq \frac{1}{\nu+1}+\frac{\sigma}{2d}$, we have
\[
\|\varphi\|_{L^{\nu+1}} \leq C\|\varphi\|_{H^{\sigma/2}}.
\]
This shows that $E(\varphi)$ is small if $\|\varphi\|_{H^{\sigma/2}}$ is small. Similarly,
\[
\frac{1}{2}\|u(t)\|^2_{\dot{H}^{\sigma/2}} = E_{\text{s}}(u(t)) -\frac{\mu}{\nu+1}\|u(t)\|^{\nu+1}_{L^{\nu+1}} \leq E_{\text{s}}(\varphi) + C \|u(t)\|^{\nu+1}_{H^{\sigma/2}},
\]
with $\nu+1>2$. This again implies that $\|u(t)\|_{H^{\sigma/2}}$ is bounded provided $\|\varphi\|_{H^{\sigma/2}}$ is small. This completes the proof.
\defendproof
\subsection{Local well-posedness in critical cases}
In this subsection, we give the proofs of Theorem $\ref{theorem local wellposedness scattering critical schrodinger sigma <2}$ and Theorem $\ref{theorem local wellposedness scattering schrodinger sigma geq 2}$. 
\paragraph{Proof of Theorem $\ref{theorem local wellposedness scattering critical schrodinger sigma <2}$.} Let us recall the following result which gives a good control for the nonlinear term. 
\begin{lem}[\cite{HongSire}] \label{lem control L nu-1 norm}
Let $\sigma \in (0,2)\backslash \{1\}$, $\nu$ be as in $(\ref{condition scattering schrodinger})$, $\gamma_{\emph{s}}$ as in $(\ref{critical exponent schrodinger})$. Then we have
\[
\|u\|^{\nu-1}_{L^{\nu-1}(\R, L^\infty)} \lesssim
\left\{ 
\begin{array}{l}
\|u\|^4_{L^4(\R,\dot{B}^{\gamma_{\emph{s}}-\gamma_{4,\infty}}_{\infty}} \|u\|^{\nu-5}_{L^\infty(\R,\dot{B}^{\gamma_{\emph{s}}}_{2})} \emph{ when } d=1, \\
\|u\|^p_{L^p(\R,\dot{B}^{\gamma_{\emph{s}}-\gamma_{p,p^\star}}_{p^\star})}\|u\|^{\nu-1-p}_{L^\infty(\R,\dot{B}^{\gamma_{\emph{s}}}_{2})}  \text{ where } \nu-1 > p >2 \text{ when } d=2, \\
\|u\|^2_{L^2(\R,\dot{B}^{\gamma_{\emph{s}}-\gamma_{2,2^\star}}_{2^\star})}\|u\|^{\nu-3}_{L^\infty(\R,\dot{B}^{\gamma_{\emph{s}}}_{2})}  \text{ when } d\geq 3,
\end{array}
\right.
\] 
where $p^\star= 2p/(p-2)$ and $2^\star= 2d/(d-2)$.
\end{lem}
This result is a slightly modification of Lemma 3.5 in \cite{HongSire}. The main difference is the power exponent in $\R^2$. The proof is similar to the one given there, thus we omit it. \newline
\textbf{Step 1.} Existence. We only treat for $d\geq 3$, the ones for $d=1, d=2$ are completely similar. Let us consider
\[
X := \Big\{ u \in L^\infty(I,H^{\gamma_{\text{s}}}) \cap L^2(I,B^{\gamma_{\text{s}}-\gamma_{2,2^\star}}_{2^\star}) \ | \ 
\|u\|_{L^\infty(I,\dot{H}^{\gamma_{\text{s}}})} \leq M, \|u\|_{L^2(I,\dot{B}^{\gamma_{\text{s}}-\gamma_{2,2^\star}}_{2^\star})} \leq N \Big\},
\]
equipped with the distance
\[
d(u,v):= \|u-v\|_{L^\infty(I,L^2)} + \|u-v\|_{L^2(I,\dot{B}^{-\gamma_{2,2^\star}}_{2^\star})},
\]
where $I=[0,T]$ and $T, M, N>0$ will be chosen later. One can check (see e.g. \cite{CazenaveWeissler} or \cite{Cazenave}) that $(X,d)$ is a complete metric space. Using the Duhamel formula
\begin{align}
\Phi(u)(t) = e^{-it\Lambda^\sigma} \varphi +i\mu \int_0^t e^{-i(t-s)\Lambda^\sigma} |u(s)|^{\nu-1} u(s)ds =: u_{\text{hom}}(t)+ u_{\text{inh}}(t), 
\label{duhamel formula}
\end{align}
the Strichartz estimate $(\ref{homogeneous full strichartz schrodinger})$ yields
\[
\|u_{\text{hom}}\|_{L^2(I,\dot{B}^{\gamma_{\text{s}}-\gamma_{2,2^\star}}_{2^\star})} \lesssim \|\varphi\|_{\dot{H}^{\gamma_{\text{s}}}}.
\]
A similar estimate holds for $\|u_{\text{hom}}\|_{L^\infty(I,\dot{H}^{\gamma_{\text{s}}})}$. We see that $\|u_{\text{hom}}\|_{L^2(I,\dot{B}^{\gamma_{\text{s}}-\gamma_{2,2^\star}}_{2^\star})} \leq \varepsilon$ for some $\varepsilon>0$ small enough which will be chosen later, provided that either $\|\varphi\|_{\dot{H}^{\gamma_{\text{s}}}}$ is small or it is satisfied some $T>0$ small enough by the dominated convergence theorem. Therefore, we can take $T=\infty$ in the first case and $T$ be this finite time in the second. On the other hand, using again $(\ref{homogeneous full strichartz schrodinger})$, we have
\[
\|u_{\text{inh}}\|_{L^2(I,\dot{B}^{\gamma_{\text{s}}-\gamma_{2,2^\star}}_{2^\star})} \lesssim \|F(u)\|_{L^1(I,\dot{H}^{\gamma_{\text{s}}})}.
\]
A same estimate holds for $\|u_{\text{inh}}\|_{L^\infty(I,\dot{H}^{\gamma_{\text{s}}})}$. Corollary $\ref{coro fractional derivatives}$ and Lemma $\ref{lem control L nu-1 norm}$ give
\begin{align}
\|F(u)\|_{L^1(I,\dot{H}^{\gamma_{\text{s}}})} \lesssim  \|u\|^{\nu-1}_{L^{\nu-1}(I,L^\infty)} \|u\|_{L^\infty(I,\dot{H}^{\gamma_{\text{s}}})} \lesssim \|u\|^2_{L^2(I,\dot{B}^{\gamma_{\text{s}}-\gamma_{2,2^\star}}_{2^\star})}\|u\|^{\nu-2}_{L^\infty(I,\dot{H}^{\gamma_{\text{s}}})}. \label{estimates sources 1} 
\end{align}
Similarly, we have
\begin{align}
\|F(u)&-F(v)\|_{L^1(I,L^2)} \lesssim \Big(\|u\|^{\nu-1}_{L^{\nu-1}(I,L^\infty)} + \|v\|^{\nu-1}_{L^{\nu-1}(I,L^\infty)} \Big) \|u-v\|_{L^\infty(I,L^2)} \label{uniqueness critical sigma <2} \\
&\lesssim   \Big(\|u\|^2_{L^2(I,\dot{B}^{\gamma_{\text{s}}-\gamma_{2,2^\star}}_{2^\star})} \|u\|^{\nu-3}_{L^\infty(I,\dot{H}^{\gamma_{\text{s}}})} + \|v\|^2_{L^2(I,\dot{B}^{\gamma_{\text{s}}-\gamma_{2,2^\star}}_{2^\star})} \|v\|^{\nu-3}_{L^\infty(I,\dot{H}^{\gamma_{\text{s}}})}\Big)\|u-v\|_{L^\infty(I,L^2)}.  \nonumber
\end{align}
This implies for all $u, v \in X$, there exists $C>0$ independent of $\varphi \in H^{\gamma_{\text{s}}}$ such that
\begin{align*}
\|\Phi(u)\|_{L^2(I,\dot{B}^{\gamma_{\text{s}}-\gamma_{2,2^\star}}_{2^\star})} &\leq \varepsilon +CN^2M^{\nu-2}, \\
\|\Phi(u)\|_{L^\infty(I,\dot{H}^{\gamma_{\text{s}}})} &\leq C\|\varphi\|_{\dot{H}^{\gamma_{\text{s}}}} +CN^2M^{\nu-2}, \\
d(\Phi(u),\Phi(v)) &\leq CN^2M^{\nu-3} d(u,v).
\end{align*}
Now by setting $N=2\varepsilon$ and $M=2C\|\varphi\|_{\dot{H}^{\gamma_{\text{s}}}}$ and choosing $\varepsilon>0$ small enough such that $CN^2M^{\nu-3} \leq \min\{1/2, \varepsilon/M \}$, we see that $X$ is stable by $\Phi$ and $\Phi$ is a contraction on $X$. By the fixed point theorem, there exists a unique solution $u \in X$ to (NLFS). Note that when $\|\varphi\|_{\dot{H}^{\gamma_{\text{s}}}}$ is small enough, we can take $T=\infty$. \newline
\textbf{Step 2.} Uniqueness. The uniqueness in $C^\infty(I,H^{\gamma_{\text{s}}}) \cap L^2(I, B^{\gamma-\gamma_{2,2^\star}}_{2^\star})$ follows as in Step 2 of the proof of Theorem $\ref{theorem local wellposedness subcritical schrodinger sigma <2}$ using $(\ref{uniqueness critical sigma <2})$. Here $\|u\|_{L^2(I,\dot{B}^{\gamma_{\text{s}}-\gamma_{2,2^\star}}_{2^\star})}$ can be small as $T$ is small.\newline
\textbf{Step 3.} Scattering. The global existence when $\|\varphi\|_{\dot{H}^{\gamma_\text{s}}}$ is small is given in Step 1. It remains to show the scattering property. Thanks to $(\ref{estimates sources 1})$, we see that
\begin{align}
\|e^{it_2\Lambda^\sigma}u(t_2)&- e^{it_1\Lambda^\sigma}u(t_1)\|_{\dot{H}^{\gamma_{\text{s}}}} =\Big\|i\mu \int_{t_1}^{t_2} e^{is\Lambda^\sigma} (|u|^{\nu-1}u)(s) ds\Big\|_{\dot{H}^{\gamma_{\text{s}}}} \nonumber \\
&\leq \|F(u)\|_{L^1([t_1,t_2], \dot{H}^{\gamma_{\text{s}}})} \lesssim  \|u\|^2_{L^2([t_1,t_2],\dot{B}^{\gamma_{\text{s}}-\gamma_{2,2^\star}}_{2^\star})}\|u\|^{\nu-2}_{L^\infty([t_1,t_2],\dot{H}^{\gamma_{\text{s}}})} \rightarrow 0 \label{scattering estimate schrodinger sigma <2(1)}
\end{align}
as $t_1, t_2 \rightarrow + \infty$. We have from $(\ref{uniqueness critical sigma <2})$ that
\begin{align}
\|e^{it_2\Lambda^\sigma}u(t_2)- e^{it_1\Lambda^\sigma}u(t_1)\|_{L^2} \lesssim  \|u\|^2_{L^2([t_1,t_2],\dot{B}^{\gamma_{\text{s}}-\gamma_{2,2^\star}}_{2^\star})}\|u\|^{\nu-3}_{L^\infty([t_1,t_2],\dot{H}^{\gamma_{\text{s}}})} \|u\|_{L^\infty([t_1,t_2],L^2)}, \label{scattering estimate schrodinger sigma <2(2)}
\end{align}
which also tends to zero as $t_1, t_2 \rightarrow +\infty$. This implies that the limit 
\[
\varphi^+:= \lim_{t \rightarrow + \infty} e^{it\Lambda^\sigma} u(t)
\]
exists in $H^{\gamma_{\text{s}}}$. Moreover, we have
\[
u(t)-e^{-it\Lambda^\sigma}\varphi^+ =-i\mu\int_{t}^{+ \infty} e^{-i(t-s)\Lambda^\sigma}F(u(s)) ds. 
\]
The unitary property of $e^{-it\Lambda^\sigma}$ in $L^2$, $(\ref{scattering estimate schrodinger sigma <2(1)})$ and $(\ref{scattering estimate schrodinger sigma <2(2)})$ imply that $\|u(t)-e^{-it\Lambda^\sigma}\varphi^+\|_{H^{\gamma_{\text{s}}}} \rightarrow 0$ when $t \rightarrow + \infty$. This completes the proof of Theorem $\ref{theorem local wellposedness scattering critical schrodinger sigma <2}$.
\defendproof
\paragraph{Proof of Theorem $\ref{theorem local wellposedness scattering schrodinger sigma geq 2}$.} The proof is similar to the one of Theorem $\ref{theorem local wellposedness scattering critical schrodinger sigma <2}$. Thus, we only give the main steps. It is easy to check that the admissible pair $(p,q)$ given in $(\ref{define p q critical})$ satisfies $\gamma_{p,q}=0=\gamma_{p',q'}+\sigma$. We next choose $n$ so that
\[
\frac{1}{q'}=\frac{1}{q}+\frac{\nu-1}{n} \text{ or } n = \frac{dq}{d-\gamma_\text{s} q}.
\] 
The Sobolev embedding gives
\begin{align}
\|u\|_{L^p(I,L^n)} \lesssim \|u\|_{L^p(I,\dot{H}^{\gamma_\text{s}})}. \label{sobolev embedding sigma geq 2}
\end{align}
\textbf{Step 1.} Existence. We will show that the functional $\Phi$ given in $(\ref{duhamel formula})$ is a contraction on 
\[
X:=\Big\{u \in L^p(I,H^{\gamma_\text{s}}_q) \ | \ \|u\|_{L^p(I,\dot{H}^{\gamma_\text{s}}_q)} \leq M \Big\},
\]
which equipped with the distance
\[
d(u,v)=\|u-v\|_{L^p(I,L^q)},
\]
where $I=[0,T]$ and $M,T>0$ to be determine. The Strichartz estimate $(\ref{usual inhomogeneous strichartz schrodinger})$ implies
\[
\|u_\text{hom}\|_{L^p(I,\dot{H}^{\gamma_\text{s}}_q)} \lesssim \|\varphi\|_{\dot{H}^{\gamma_\text{s}}}.
\]
This shows that $\|u_\text{hom}\|_{L^p(I,\dot{H}^{\gamma_\text{s}}_q)} \leq \varepsilon$ for some $\varepsilon>0$ small enough provided that $T$ is small or $\|\varphi\|_{\dot{H}^{\gamma_\text{s}}}$ is small. Similarly, we have
\[
\|u_\text{inh}\|_{L^p(I,\dot{H}^{\gamma_\text{s}}_q)} \lesssim \|F(u)\|_{L^{p'}(I,\dot{H}^{\gamma_\text{s}}_{q'})}.
\]
It follows from Corollary $\ref{coro fractional derivatives}$, the choice of $n$ and $(\ref{sobolev embedding sigma geq 2})$ that
\begin{align}
\|F(u)\|_{L^{p'}(I,\dot{H}^{\gamma_\text{s}}_{q'})} &\lesssim \|u\|^\nu_{L^p(I,\dot{H}^{\gamma_\text{s}}_q)}, \label{scattering estimate sigma geq 2} \\
\|F(u)-F(v)\|_{L^{p'}(I,L^{q'})} &\lesssim \Big(\|u\|^{\nu-1}_{L^p(I,\dot{H}^{\gamma_\text{s}}_q)} +\|v\|^{\nu-1}_{L^p(I,\dot{H}^{\gamma_\text{s}}_q)} \Big) \|u-v\|_{L^p(I,L^q)}. \label{uniqueness critical sigma geq 2}
\end{align}
Thus, the Strichartz estimate $(\ref{usual inhomogeneous strichartz schrodinger})$ implies for all $u, v \in X$, there exists $C$ independent of $\varphi \in H^{\gamma_\text{s}}$ such that
\begin{align*}
\|\Phi(u)\|_{L^p(I,\dot{H}^{\gamma_\text{s}}_q)} &\leq \varepsilon + C M^\nu, \nonumber \\
d(\Phi(u),\Phi(v)) &\leq C M^{\nu-1}d(u,v).
\end{align*}
If we choose $\varepsilon, M>0$ small so that
\[
C M^{\nu-1}\leq \frac{1}{2}, \quad \varepsilon + \frac{M}{2} \leq M,
\]
then $X$ is stable by $\Phi$ and $\Phi$ a contraction on $X$. Using the argument as in Step 1 of the proof of Theorem $\ref{theorem local wellposedness subcritical schrodinger sigma geq 2}$, we obtain the existence of solution $u \in C(I,H^{\gamma_\text{s}})\cap L^p(I,H^{\gamma_\text{s}}_q)$ to (NLFS). Note that when $\|\varphi\|_{\dot{H}^{\gamma_\text{s}}}$ is small, we can take $T=\infty$.\newline
\textbf{Step 2.} Uniqueness. It follows easily from $(\ref{uniqueness critical sigma geq 2})$ by the same argument given in Step 2 of the proof of Theorem $\ref{theorem local wellposedness subcritical schrodinger sigma <2}$ using $(\ref{uniqueness critical sigma geq 2})$. \newline
\textbf{Step 3.} Scattering. The global existence when $\|\varphi\|_{\dot{H}^{\gamma_\text{s}}}$ is small follows from Step 1. The scattering is treated similarly as in Step 3 of the proof of Theorem $\ref{theorem local wellposedness scattering critical schrodinger sigma <2}$. The main point is to show
\begin{align}
\|e^{it_2\Lambda^\sigma} u(t_2)-e^{it_1\Lambda^\sigma}u(t_1)\|_{H^{\gamma_\text{s}}} \rightarrow 0 \label{scattering critical schrodinger sigma geq 2}
\end{align}
as $t_1, t_2 \rightarrow +\infty$. To do so, we use the adjoint estimate to the homogeneous Strichartz estimate, namely $\varphi \in L^2 \mapsto e^{-it\Lambda^\sigma} \varphi \in L^p(\R,L^q)$ to get
\begin{align*}
\|e^{it_2\Lambda^\sigma} u(t_2)-e^{it_1\Lambda^\sigma}u(t_1)\|_{\dot{H}^{\gamma_\text{s}}} & =\Big\|
i\mu \int_{t_1}^{t_2} e^{is\Lambda^\sigma} (|u|^{\nu-1}u)(s)ds\Big\|_{\dot{H}^{\gamma_\text{s}}} \\
&= \Big\| \int_\R \Lambda^{\gamma_\text{s}} e^{is\Lambda^\sigma} (\mathds{1}_{[t_1,t_2]}|u|^{\nu-1}u)(s)ds \Big\|_{L^2} \\
&\lesssim \|F(u)\|_{L^{p'}([t_1,t_2],\dot{H}^{\gamma_\text{s}}_{q'})}.
\end{align*}
Similarly,
\[
\|e^{it_2\Lambda^\sigma} u(t_2)-e^{it_1\Lambda^\sigma}u(t_1)\|_{L^2} \lesssim  \|F(u)\|_{L^{p'}([t_1,t_2],L^{q'})}.
\]
Using $(\ref{scattering estimate sigma geq 2})$ and $(\ref{uniqueness critical sigma geq 2})$, we get $(\ref{scattering critical schrodinger sigma geq 2})$. The proof is complete.
\defendproof 
\section{Nonlinear fractional wave equations}
\setcounter{equation}{0}
\subsection{Local well-posedness in subcritical cases}
In this subsection, we will give the proofs of Theorem $\ref{theorem local wellposedness subcritical wave}$ and Theorem $\ref{theorem local wellposedness wave H sigma subcritical}$.
\paragraph{Proof of Theorem $\ref{theorem local wellposedness subcritical wave}$.} 
The proof is very close to the one of Theorem $\ref{theorem local wellposedness subcritical schrodinger sigma <2}$. Let $(p,q)$ be the fractional pair in the proof of Theorem $\ref{theorem local wellposedness subcritical schrodinger sigma <2}$. \newline
\textbf{Step 1.} Existence. We will solve (NLFW) in
\[
Y:= \Big\{ v \in C(I,H^\gamma) \cap C^1(I, H^{\gamma-\sigma}) \cap L^p(I,H^{\gamma - \gamma_{p,q}}_q) \ | \  \|[v]\|_{L^\infty(I,H^\gamma)}+ \|v\|_{L^p(I, H^{\gamma-\gamma_{p,q}}_q)} \leq M \Big\},
\]
equipped with the distance
\[
d(v,w) := \|[v-w]\|_{L^\infty(I,L^2)} + \|v-w\|_{L^p(I,H^{- \gamma_{p,q}}_q)},
\]
where $I=[0,T]$ and $T, M>0$ will be chosen later. The persistence of regularity implies that $(Y,d)$ is a complete metric space. By the Duhamel formula, it suffices to prove that the functional 
\begin{align}
\Psi(v)(t)= \cos (t\Lambda^\sigma) \varphi + \frac{\sin (t\Lambda^\sigma)}{\Lambda^\sigma} \phi - \mu \int_0^t \frac{\sin((t-s)\Lambda^\sigma)}{\Lambda^\sigma}|v(s)|^{\nu-1}v(s)ds \label{duhamel formula wave}
\end{align}
is a contraction on $Y$. The local Strichartz estimates $(\ref{local strichartz wave})$ imply
\begin{align*}
\|[\Psi(v)]\|_{L^\infty(I,H^\gamma)}+ \|\Psi(v)\|_{L^p(I, H^{\gamma-\gamma_{p,q}}_q)} & \lesssim \|[v](0)\|_{H^\gamma}+ \|F(v)\|_{L^1(I,H^{\gamma-\sigma})} \\
&\lesssim \|[v](0)\|_{H^\gamma}+ \|F(v)\|_{L^1(I,H^{\gamma})}, 
\end{align*}
where $F(v) = |v|^{\nu-1} v$. 
As in the proof of Theorem $\ref{theorem local wellposedness subcritical schrodinger sigma <2}$, Corollary $\ref{coro fractional derivatives}$ implies
\[
\|F(v)\|_{L^1(I,H^\gamma)} \lesssim  T^{1-\frac{\nu-1}{p}} \|v\|^{\nu-1}_{L^p(I,L^\infty)} \|v\|_{L^\infty(I,H^\gamma)}.
\]
Similarly, 
\begin{align}
\|F(v)-F(w)\|_{L^1(I,L^2)} \lesssim  T^{1-\frac{\nu-1}{p}} \Big(\|v\|^{\nu-1}_{L^p(I,L^\infty)} + \|w\|^{\nu-1}_{L^p(I,L^\infty)}\Big) \|v-w\|_{L^\infty(I,L^2)}. \label{uniqueness subcritical wave}
\end{align}
The Sobolev embedding $L^p(I,H^{\gamma-\gamma_{p,q}}_q) \subset L^p(I, L^\infty)$ then implies that
\[
\|[\Psi(v)]\|_{L^\infty(I,H^\gamma)}+ \|\Psi(v)\|_{L^p(I, H^{\gamma-\gamma_{p,q}}_q)} \lesssim \|[v](0)\|_{H^\gamma}+  T^{1-\frac{\nu-1}{p}} \|v\|^{\nu-1}_{L^p(I,H^{\gamma-\gamma_{p,q}}_q)} \|v\|_{L^\infty(I,H^\gamma)},
\]
and 
\[
d(\Psi(v),\Psi(w)) \lesssim T^{1-\frac{\nu-1}{p}} \Big(\|v\|^{\nu-1}_{L^p(I,H^{\gamma-\gamma_{p,q}}_q)} + \|w\|^{\nu-1}_{L^p(I,H^{\gamma-\gamma_{p,q}}_q)}\Big)d(v,w).
\]
Therefore, for all $v, w \in Y$, there exists a constant $C>0$ independent of $\varphi,\phi$ such that 
\[
\|[\Psi(v)]\|_{L^\infty(I,H^\gamma)}+ \|\Psi(v)\|_{L^p(I, H^{\gamma-\gamma_{p,q}}_q)} \leq C\|[v](0)\|_{H^\gamma}+  CT^{1-\frac{\nu-1}{p}} M^\nu,
\]
and 
\[
d(\Psi(v),\Psi(w)) \leq C T^{1-\frac{\nu-1}{p}} M^{\nu-1} d(v,w).
\]
Setting $M=2C \|[v](0)\|_{H^\gamma}$ and choosing $T>0$ small enough so that $CT^{1-\frac{\nu-1}{p}} M^{\nu-1} \leq \frac{1}{2}$, we see that $Y$ is stable by $\Psi$ and $\Psi$ is a contraction on $Y$. By the fixed point theorem, there exists a unique solution $v \in Y$ to (NLFW). \newline
\text{Step 2.} Uniqueness. The uniqueness of solution $v \in C(I, H^\gamma) \cap C^1(I, H^{\gamma-\sigma}) \cap L^p(I, L^\infty)$ follows as in the proof of Theorem $\ref{theorem local wellposedness subcritical schrodinger sigma <2}$ using $(\ref{uniqueness subcritical wave})$. \newline
\textbf{Step 3.} The blowup alternative follows easily since the time of existence depends only on $\|[v](0)\|_{H^\gamma}$. \newline
\textbf{Step 4.} The continuous dependence is similar to that of Theorem $\ref{theorem local wellposedness subcritical schrodinger sigma <2}$. \defendproof
\paragraph{Proof of Theorem $\ref{theorem local wellposedness wave H sigma subcritical}$.} 
1. Let us firstly consider Item 1. We note (see Remark $\ref{rem admissible condition}$) that under the assumptions $(\ref{assumption sigma <2 subcritical 1}), (\ref{assumption sigma <2 subcritical 2})$ and $(\ref{assumption sigma <2 subcritical 3})$ (see Remark $\ref{rem admissible condition}$), the pair $(p,q)$ given in $(\ref{define p q sigma subcritical wave sigma <2})$ is admissible satisfying $\gamma_{p,q}=\sigma=\gamma_{1,2}+2\sigma$ and $1-\nu/p>0$. Consider now 
\[
Y:=\Big\{ v \in C(I, \dot{H}^\sigma) \cap C^1(I,L^2) \cap L^p(I, L^q) \ | \ \|[v]\|_{L^\infty(I,\dot{H}^\sigma)} + \|v\|_{L^p(I,L^q)} \leq M \Big\},
\] 
equipped with the distance
\[
d(v,w):= \|[v-w]\|_{L^\infty(I,\dot{H}^\sigma)} + \|v-w\|_{L^p(I, L^q)},
\]
where $I=[0,T]$ and $M>0$ will be chosen later. We will prove that the functional $(\ref{duhamel formula wave})$ is a contraction on $Y$. The Strichartz estimate $(\ref{usual inhomogeneous strichartz wave})$ implies
\begin{align}
\|[\Psi(v)]\|_{L^\infty(I,\dot{H}^\sigma)} + \|\Psi(v)\|_{L^p(I,L^q)} &\lesssim \|[v](0)\|_{\dot{H}^\sigma} + \|F(v)\|_{L^1(I, L^2)} =  \|[v](0)\|_{\dot{H}^\sigma}  + \|v\|^{\nu}_{L^\nu(I, L^{2\nu})} \nonumber \\
&\lesssim \|[v](0)\|_{\dot{H}^\sigma}  + T^{1-\frac{\nu}{p}}\|v\|^{\nu}_{L^p(I, L^q)}. \nonumber
\end{align}
Similarly, 
\begin{align}
\|F(v)-F(w)\|_{L^1(I, L^2)}& \lesssim \Big(\|v\|^{\nu-1}_{L^\nu(I,L^{2\nu})} + \|w\|^{\nu-1}_{L^\nu(I,L^{2\nu})}\Big) \|v-w\|_{L^\nu(I,L^{2\nu})} \nonumber \\
& \lesssim T^{1-\frac{\nu}{p}}\Big(\|v\|^{\nu-1}_{L^p(I, L^q)} + \|v\|^{\nu-1}_{L^p(I,L^q)} \Big) \|v-w\|_{L^p(I, L^q)}. \label{continuous dependence wave}
\end{align}
This implies that for all $v, w \in Y$, there exists $C>0$ independent of $(\varphi,\phi)\in \dot{H}^\sigma \times L^2$ such that,
\begin{align}
\|[\Psi(v)]\|_{L^\infty(I,\dot{H}^\sigma)} +\|\Psi(v)\|_{L^p(I,L^q)} &\leq C\|[v](0)\|_{\dot{H}^\sigma}+ C T^{1-\frac{\nu}{p}} M^\nu, \nonumber \\
d(\Psi(v),\Psi(w)) &\leq C T^{1-\frac{\nu}{p}} M^{\nu-1}d(v,w). \nonumber
\end{align}
By setting $M=2C \|[v](0)\|_{\dot{H}^\sigma}$, choosing $T>0$ small enough so that $CT^{1-\frac{\nu}{p}} M^{\nu-1} \leq \frac{1}{2}$ and arguing as in the proof of Theorem $\ref{theorem local wellposedness subcritical wave}$, we have the existence and uniqueness of solution $v \in C(I, \dot{H}^\sigma) \cap C^1(I,L^2) \cap L^p(I, L^q)$. The blowup alternative is immediate since the time of existence only depends on $\|[v](0)\|_{\dot{H}^\sigma}$. Finally, the continuous dependence is proved by using $(\ref{continuous dependence wave})$. \newline
2. The proof of Item 2 is similar, thus we only give the main steps. It is easy to see that under the assumption $(\ref{assumption sigma subcritical sigma geq 2})$, the pair $(p,q)$ defined in $(\ref{define p q sigma subcritical wave sigma geq 2})$ is admissible and $\gamma_{p,q}=\sigma$. Since $\nu \in [d\sigma^*/(d+\sigma),\sigma^*)$, we see that $q/\nu \in (1,2]$. This allows to choose $b \in [2,\infty]$ so that $b'=q/\nu$. We next choose $a \in [2,\infty]$ such that $(a,b)$ is admissible and $\gamma_{a,b}=-\gamma_{a',b'}-\sigma = 0$ or $\gamma_{a',b'}+2\sigma=\sigma$. Thanks to the fact that $\nu<\sigma^*$, we see that
\[
\frac{1}{a'}-\frac{\nu}{p}>0.
\] 
This shows that $\frac{1}{a'}=\frac{1}{p}+\frac{\nu-1}{m}$ with
\[
\frac{\nu-1}{m}>\frac{\nu-1}{p}.
\]
We will prove that $\Psi$ is a contraction on
\[
Y:=\Big\{ v \in  v \in C(I, \dot{H}^\sigma) \cap C^1(I,L^2) \cap L^p(I, L^q) \ | \ \|[v]\|_{L^\infty(I,\dot{H}^\sigma)} + \|v\|_{L^p(I,L^q)} \leq M \Big\},
\] 
equipped with the distance
\[
d(v,w):= \|[v-w]\|_{L^\infty(I,\dot{H}^\sigma)} + \|v-w\|_{L^p(I, L^q)}.
\]
The Strichartz estimate $(\ref{usual inhomogeneous strichartz wave})$ implies
\begin{align}
\|[\Psi(v)]\|_{L^\infty(I,\dot{H}^\sigma)} + \|\Psi(v)\|_{L^p(I,L^q)} &\lesssim \|[v](0)\|_{\dot{H}^\sigma} + \|F(v)\|_{L^{a'}(I, L^{b'})} \nonumber \\
&=  \|[v](0)\|_{\dot{H}^\sigma}  + \|v\|^{\nu-1}_{L^m(I, L^q)}\|v\|_{L^p(I,L^q)} \nonumber \\
&\lesssim \|[v](0)\|_{\dot{H}^\sigma}  + T^{\frac{\nu-1}{m}-\frac{\nu-1}{p}}\|v\|^\nu_{L^p(I, L^q)}. \nonumber
\end{align}
Similarly, 
\begin{align*}
\|F(v)-F(w)\|_{L^{a'}(I, L^{b'})}& \lesssim \Big(\|v\|^{\nu-1}_{L^m(I,L^q)} + \|w\|^{\nu-1}_{L^m(I,L^q)}\Big) \|v-w\|_{L^p(I,L^q)} \nonumber \\
& \lesssim T^{\frac{\nu-1}{m}-\frac{\nu-1}{p}} \Big(\|v\|^{\nu-1}_{L^p(I, L^q)} + \|v\|^{\nu-1}_{L^p(I,L^q)} \Big)  \|v-w\|_{L^p(I, L^q)}. \nonumber 
\end{align*}
This implies that for all $v, w \in Y$, there exists $C>0$ independent of $(\varphi,\phi)\in \dot{H}^\sigma \times L^2$ such that,
\begin{align}
\|[\Psi(v)]\|_{L^\infty(I,\dot{H}^\sigma)} +\|\Psi(v)\|_{L^p(I,L^q)} &\leq C\|[v](0)\|_{\dot{H}^\sigma}+ C T^{\frac{\nu-1}{m}-\frac{\nu-1}{p}} M^\nu, \nonumber \\
d(\Psi(v),\Psi(w)) &\leq C T^{\frac{\nu-1}{m}-\frac{\nu-1}{p}} M^{\nu-1}d(v,w). \nonumber
\end{align}
The conclusion is similar as in Item 1. The proof is now complete.
\defendproof \newline
\begin{rem} \label{rem admissible condition}
Let us give some comments on the assumptions $(\ref{assumption sigma <2 subcritical 1}), (\ref{assumption sigma <2 subcritical 2})$ and $(\ref{assumption sigma <2 subcritical 3})$.  
In order to make $(p,q)$ defined in $(\ref{define p q sigma subcritical wave sigma <2})$ to be a admissible satisfying $\gamma_{p,q}=\sigma=\gamma_{1,2}+2\sigma$ and $1-\nu/p>0$, we need the following conditions: \newline
- A first condition is $(d-2\sigma)\nu>d$ which ensures $p$ is a positive number. \newline
- We next one is $p\geq 4$ when $d=1$ and $p\geq 2$ when $d\geq 2$. Thus $(2-5\sigma)\nu \leq 2$ when $d=1$ and $(d-3\sigma)\nu \leq d$ when $d\geq 2$. \newline
- We also need $\frac{2}{p}+\frac{d}{q} \leq \frac{d}{2}$ which implies $(2d-4\sigma-d\sigma)\nu \leq 2d-d\sigma$. When $d=1$, we have $(2-5\sigma)\nu \leq 2-\sigma$. \newline
- Condition $\gamma_{p,q}=\sigma=\gamma_{1,2}+2\sigma$ is easy to check. \newline
- Finally, we have $(d-2\sigma)\nu< d+2\sigma$ which yields $1-\nu/p>0$. \newline
Therefore, we need
\[
\left\{
\begin{array}{l}
(1-2\sigma)\nu >1 \\
(1-2\sigma) \nu <1+2\sigma \\
(2-5\sigma) \nu \leq 2-\sigma
\end{array} 
\right.
\text{ when } d=1 \text{ and }
\left\{
\begin{array}{l}
(d-2\sigma) \nu >d \\
(d-2\sigma) \nu < d+2\sigma \\
(d-3\sigma)\nu \leq d \\
(2d-4\sigma-d\sigma) \nu \leq 2d-d\sigma
\end{array}
\right. \text{ when } d\geq 2.
\]
One can solve easily the above systems of inequalities and obtain $(\ref{assumption sigma <2 subcritical 1})$, $(\ref{assumption sigma <2 subcritical 2})$ and $(\ref{assumption sigma <2 subcritical 3})$.
\end{rem}
\subsection{Local well-posedness in critical cases}
In this subsection, we will give the proofs of Theorem $\ref{theorem local wellposedness critical scattering small data wave}$ and Theorem $\ref{theorem local wellposedness wave H sigma critical}$. \paragraph{Proof of Theorem $\ref{theorem local wellposedness critical scattering small data wave}$.}
1. Let us treat the first case $(\ref{assumption critical wave 1})$. Consider 
\begin{multline}
Y:=\Big\{ v \in  C(I, \dot{H}^{\gamma_{\text{w}}}) \cap C^1(I,\dot{H}^{\gamma_{\text{w}}-\sigma}) \cap L^p(I,L^p) \cap L^a(I, \dot{H}^{\gamma_{\text{w}}-\frac{\sigma}{2}}_a) \Big. \\
\Big. \|[v]\|_{L^\infty(I,\dot{H}^{\gamma_{\text{w}}})} \leq M, \|v\|_{L^p(I,L^p)} + \|v\|_{L^a(I, \dot{H}^{\gamma_{\text{w}}-\frac{\sigma}{2}}_a)} \leq N \Big\} \nonumber
\end{multline}
equipped with the distance
\[
d(v,w):= \|[v-w]\|_{L^\infty(I,\dot{H}^{\gamma_{\text{w}}})} + \|v-w\|_{L^p(I,L^p)} + \|v-w\|_{L^a(I, \dot{H}^{\gamma_{\text{w}}-\frac{\sigma}{2}}_a)}, 
\]
where $(p,a)$ given in $(\ref{define p a})$, $I=[0,T]$ and $T, M, N>0$ will be chosen later. Using the Duhamel's formula, it suffices to show that the functional
\[
\Psi(v)(t)= \cos (t\Lambda^\sigma) \varphi + \frac{\sin (t\Lambda^\sigma)}{\Lambda^\sigma} \phi - \mu \int_0^t \frac{\sin((t-s)\Lambda^\sigma)}{\Lambda^\sigma}|v(s)|^{\nu-1}v(s)ds =: v_{\text{hom}}(t)+v_{\text{inh}}(t),
\]
is a contraction on $Y$, where $v_{\text{hom}}(t)$ is the sum of two first terms and $v_{\text{inh}}(t)$ is the last term. It is easy to check that under the assumptions $(\ref{assumption critical wave 1})$, $(p,p)$ and $(a,a)$ are admissible with $\gamma_{p,p}=\gamma_{\text{w}}$ and $\gamma_{a,a} =\sigma/2$. The Strichartz estimate $(\ref{usual inhomogeneous strichartz wave})$ then implies
\begin{align}
\|v_{\text{hom}}\|_{L^p(I,L^p)} + \|v_{\text{hom}}\|_{L^a(I, \dot{H}^{\gamma_{\text{w}}-\frac{\sigma}{2}}_a)} \lesssim \|[v](0)\|_{\dot{H}^{\gamma_{\text{w}}}}. \label{homogeneous part wave critical}
\end{align}
Thus the left hand side of $(\ref{homogeneous part wave critical})$ can be taken smaller than $\varepsilon$ for some $\varepsilon>0$ small enough provided that either $\|[v](0)\|_{\dot{H}^{\gamma_{\text{w}}}}$ is small or it is true for some $T>0$ small enough by the dominated convergence. On the other hand, the homogeneous Sobolev embedding with the fact that $\gamma_{\text{w}}-\sigma/2 \geq 0$ implies $L^p(I,\dot{H}^{\gamma_{\text{w}}-\frac{\sigma}{2}}_q) \subset L^p(I,L^p)$ where $d/q=d/p + (\gamma_{\text{w}}-\sigma/2)$. For such $q$, we see that $(p,q)$ is admissible satisfying
\[
\gamma_{p,q} = \frac{\sigma}{2}=\gamma_{a,a} =\gamma_{a',a'} +2\sigma.
\]
The Sobolev embedding and Strichartz estimate $(\ref{usual inhomogeneous strichartz wave})$ then yield
\[
\|v_{\text{inh}}\|_{L^p(I,L^p)} + \|v_{\text{inh}}\|_{L^a(I, \dot{H}^{\gamma_{\text{w}}-\frac{\sigma}{2}}_a)} \lesssim \|F(v)\|_{L^{a'}(I, \dot{H}^{\gamma_{\text{w}}-\frac{\sigma}{2}}_{a'})}.
\]
Using $(\ref{assumption smoothness critical wave})$ and the fact that $\frac{1}{a'} =\frac{1}{a}+\frac{\nu-1}{p}$, Corollary $\ref{coro fractional derivatives}$ gives
\[
\|F(v)\|_{L^{a'}(I, \dot{H}^{\gamma_{\text{w}}-\frac{\sigma}{2}}_{a'})} \lesssim \|v\|^{\nu-1}_{L^p(I,L^p)} \|v\|_{L^a(I, \dot{H}^{\gamma_{\text{w}}-\frac{\sigma}{2}}_a)}. 
\]
Similarly,
\begin{align}
\|F(v)&-F(w)\|_{L^{a'}(I, \dot{H}^{\gamma_{\text{w}}-\frac{\sigma}{2}}_{a'})}  \lesssim  (\|v\|^{\nu-1}_{L^p(I,L^p)}+\|w\|^{\nu-1}_{L^p(I,L^p)})\|u-v\|_{L^a(I,\dot{H}^{\gamma_\text{w}-\frac{\sigma}{2}}_a)} \nonumber \\
&+(\|v\|^{\nu-2}_{L^p(I,L^p)}+\|w\|^{\nu-2}_{L^p(I,L^p)})(\|v\|_{L^a(I,\dot{H}^{\gamma_\text{w}-\frac{\sigma}{2}}_a)} + \|w\|_{L^a(I,\dot{H}^{\gamma_\text{w}-\frac{\sigma}{2}}_a)}) \|v-w\|_{L^p(I,L^p)}. \label{uniqueness critical wave}
\end{align}
Similarly, by rewriting $\gamma_{\text{w}}=\gamma_{\text{w}}-\frac{\sigma}{2}+ \gamma_{a,a}$, the Strichartz estimate $(\ref{usual inhomogeneous strichartz wave})$ also gives
\begin{align*}
\|[\Psi(v)]\|_{L^\infty(I,\dot{H}^{\gamma_{\text{w}}})}  
\lesssim \|[v](0)\|_{\dot{H}^{\gamma_{\text{w}}}}+ \|v\|^{\nu-1}_{L^p(I,L^p)} \|v\|_{L^a(I, \dot{H}^{\gamma_{\text{w}}-\frac{\sigma}{2}}_a)}. \nonumber
\end{align*}
This implies for all $v, w \in Y$, there exists $C>0$ independent of $(\varphi, \phi)\in \dot{H}^{\gamma_\text{w}} \times \dot{H}^{\gamma_\text{w}-\sigma}$ such that
\begin{align*}
\|\Psi(v)\|_{L^p(I,L^p)} + \|\Psi(v)\|_{L^a(I, \dot{H}^{\gamma_{\text{w}}-\frac{\sigma}{2}}_a)} &\leq \varepsilon +CN^\nu, \\
\|[\Psi(v)]\|_{L^\infty(I, \dot{H}^{\gamma_{\text{w}}})} &\leq C\|[v](0)\|_{\dot{H}^{\gamma_{\text{w}}}} +CN^\nu, \\
d(\Psi(v),\Psi(w)) &\leq CN^{\nu-1} d(v,w).
\end{align*}
Now by setting $N=2\varepsilon$ and $M=2C\|[v](0)\|_{\dot{H}^{\gamma_{\text{w}}}}$ and choosing $\varepsilon>0$ small enough (provided either $T$ is small or $\|[v](0)\|_{\dot{H}^{\gamma_{\text{w}}}}$ is small) such that
\[
CN^{\nu} \leq \min \Big\{ \varepsilon, C\|[v](0)\|_{\dot{H}^{\gamma_{\text{w}}}}\Big\}, \quad CN^{\nu-1} \leq \frac{1}{2},
\]
we see that $Y$ is stable by $\Psi$ and $\Psi$ is a contraction on $Y$. By the fixed point theorem, there exists a unique solution $v \in Y$ to (NLFW). Note that when $\|[v](0)\|_{\dot{H}^{\gamma_{\text{w}}}}$ is small enough, we can take $T=\infty$. The uniqueness in $C(I, \dot{H}^{\gamma_{\text{w}}}) \cap C^1(I,\dot{H}^{\gamma_{\text{w}}-\sigma}) \cap L^p(I,L^p) \cap L^a(I, \dot{H}^{\gamma_{\text{w}}-\frac{\sigma}{2}}_a)$ follows as in Theorem $\ref{theorem local wellposedness subcritical schrodinger sigma <2}$ by using $(\ref{uniqueness critical wave})$. Here $\|v\|_{L^p(I,L^p)}$ and  $\|v\|_{L^a(I, \dot{H}^{\gamma_{\text{w}}-\frac{\sigma}{2}}_a)}$ can be small as $T$ is small.\newline
\indent We now prove the scattering property of the global solution. Let us denote
\[
V(t) := \left[\begin{array}{c}
v(t) \\ 
\partial_t v(t)
\end{array}\right], \quad A := \left(\begin{array}{cc}
0 & 1 \\
-\Lambda^{2\sigma} & 0
\end{array}\right), \quad G(V(t)) := \left[\begin{array}{c}
0 \\
F(v(t))
\end{array}\right].
\]
The (NLFW) can be written as
\[
\partial_t V(t) - A V(t) = G(V(t)),
\]
or
\[
V(t)= e^{tA} V(0)+ \int_{0}^{t} e^{(t-s)A} G(V(s))ds,
\]
where 
\[
e^{tA}:= \left( \begin{array}{cc}
\cos t\Lambda^\sigma & \frac{\sin t\Lambda^\sigma}{\Lambda^\sigma} \\
-\Lambda^\sigma \sin t\Lambda^\sigma & \cos t\Lambda^\sigma
\end{array}\right).
\]
The adjoint estimates of $e^{\pm it\Lambda^\sigma}: L^a([t_1,t_2],L^a)  \rightarrow \dot{H}^{\gamma_{a,a}}$ with $\gamma_{a,a}=\sigma/2$ imply
\begin{align*}
\Big\| \int_{t_1}^{t_2} e^{\pm is\Lambda^\sigma} F(v(s)) ds \Big\|_{\dot{H}^{\gamma_{\text{w}}-\sigma}} &= \Big\| \int_{t_1}^{t_2} \Lambda^{-\frac{\sigma}{2}} e^{\pm is\Lambda^\sigma}  \Lambda^{\gamma_{\text{w}}-\frac{\sigma}{2}} F(v(s)) ds \Big\|_{L^2} \\
& \lesssim \|F(v)\|_{L^{a'}([t_1,t_2], \dot{H}^{\gamma_{\text{w}}-\frac{\sigma}{2}}_{a'})} \\
& \lesssim \|v\|^{\nu-1}_{L^p([t_1,t_2],L^p)} \|v\|_{L^a([t_1,t_2], \dot{H}^{\gamma_{\text{w}}-\frac{\sigma}{2}}_a)} \rightarrow 0
\end{align*}
as $t_1,t_2 \rightarrow + \infty$. This implies that
\begin{align}
\|[e^{-t_2A} V(t_2)-e^{-t_1A} V(t_1)]\|_{\dot{H}^{\gamma_{\text{w}}}}= \Big\| \Big[\int_{t_1}^{t_2} e^{-sA} G(V(s))ds \Big] \Big\|_{\dot{H}^{\gamma_{\text{w}}}} \rightarrow 0 \label{scattering estimate wave}
\end{align}
as $t_1,t_2 \rightarrow + \infty$. Therefore, the limit 
\[
V^+(0):= \lim_{t\rightarrow + \infty} e^{-tA} V(t)
\]
exists in $\dot{H}^{\gamma_{\text{w}}} \times \dot{H}^{\gamma_{\text{w}}-\sigma}$. We also have
\[
V(t)-e^{tA}V^+(0) = - \int_{t}^{+ \infty} e^{(t-s)A} G(V(s)) ds. 
\]
Using the unitary property of $e^{\pm it\Lambda^\sigma}$ in $L^2$ and $(\ref{scattering estimate wave})$, we have $\|[V(t)-e^{tA}V^+(0)]\|_{\dot{H}^{\gamma_{\text{w}}}} \rightarrow 0$ as $t \rightarrow + \infty$. This completes the proof of Item 1.\newline
\indent 2. We next consider the case $(\ref{assumption critical wave 2})$. The proof is similar as above, thus we only give the main steps. We will solve (NLFW) in 
\[
Y:= \Big\{v \in C(I,\dot{H}^{\gamma_{\text{w}}}) \cap C^1(I,\dot{H}^{\gamma_{\text{w}}-\sigma}) \cap L^p(I,L^p)\ | \  \|[v]\|_{L^\infty(I,\dot{H}^{\gamma_{\text{w}}})} \leq M, \|v\|_{L^p(I,L^p)} \leq N \Big\}, 
\]
equipped with the distance
\[
d(v,w):= \|[v-w]\|_{L^\infty(I,\dot{H}^{\gamma_{\text{w}}})} + \|v-w\|_{L^p(I,L^p)},
\]
where $p$ is as in Item 1. It is easy to check that under the assumption $(\ref{assumption critical wave 2})$, $(p,p)$ and $(b,b)$ are admissible and
\[
\gamma_{p,p}=\gamma_\text{w}=\gamma_{b',b'}+2\sigma,
\]
where $b'=p/\nu$. The Strichartz estimate $(\ref{usual inhomogeneous strichartz wave})$ implies
$\|v_{\text{hom}}\|_{L^p(I,L^p)} \lesssim \|[v](0)\|_{\dot{H}^{\gamma_{\text{w}}}}$. Therefore, $\|v_{\text{hom}}\|_{L^p(I,L^p)} \leq \varepsilon$ for some $\varepsilon>0$ small enough provided that $T$ is small or $\|[v](0)\|_{\dot{H}^{\gamma_{\text{w}}}}$ is small. Similarly, 
\[
\|v_{\text{inh}}\|_{L^p(I,L^p)} \lesssim \|F(v)\|_{L^{b'}(I,L^{b'})}\lesssim \|v\|^\nu_{L^p(I,L^p)},
\]
where the last inequality follows from the H\"older inequality with the fact that
\[
\frac{1}{b'} = \frac{1}{p} + \frac{\nu-1}{p}.
\]
We also have from $(\ref{usual inhomogeneous strichartz wave})$ that 
\begin{align}
\|F(v)-F(w)\|_{L^{b'}(I,L^{b'})}\lesssim \Big(\|v\|^{\nu-1}_{L^p(I,L^p)}+\|w\|^{\nu-1}_{L^p(I,L^p)} \Big) \|v-w\|_{L^p(I,L^p)}. \label{uniqueness critical wave 2}
\end{align}
This implies for all $v, w \in Y$, there exists $C>0$ independent of $(\varphi, \phi)\in \dot{H}^{\gamma_\text{w}} \times \dot{H}^{\gamma_\text{w}-\sigma}$ such that
\begin{align*}
\|\Psi(v)\|_{L^p(I,L^p)}  &\leq \varepsilon +CN^\nu, \\
\|[\Psi(v)]\|_{L^\infty(I, \dot{H}^{\gamma_{\text{w}}})} &\leq C\|[v](0)\|_{\dot{H}^{\gamma_{\text{w}}}} +CN^\nu, \\
d(\Psi(v),\Phi(w)) &\leq CN^{\nu-1} d(v,w).
\end{align*}
Now by setting $N=2\varepsilon$ and $M=2C\|[v](0)\|_{\dot{H}^{\gamma_{\text{w}}}}$ and choosing $\varepsilon>0$ small enough, we have the existence of solution $v \in Y$ to (NLFW). The uniqueness in $C(I, \dot{H}^{\gamma_{\text{w}}}) \cap C^1(I,\dot{H}^{\gamma_{\text{w}}-\sigma}) \cap L^p(I,L^p)$ follows as in Theorem $\ref{theorem local wellposedness subcritical schrodinger sigma <2}$ by using $(\ref{uniqueness critical wave 2})$. Here $\|v\|_{L^p(I,L^p)}$ can be small as $T$ is small.\newline
\indent Using the adjoint Strichartz estimates with the fact that $\gamma_{b,b}=-\gamma_{b',b'} -\sigma = -\gamma_{\text{w}}+\sigma$, we have
\begin{align*}
\Big\| \int_{t_1}^{t_2} e^{\pm is\Lambda^\sigma} F(v(s)) ds \Big\|_{\dot{H}^{\gamma_{\text{w}}-\sigma}} &= \Big\|\int_{t_1}^{t_2} \Lambda^{\gamma_{\text{w}}-\sigma} e^{\pm is\Lambda^\sigma} F(v(s)) ds\Big\|_{L^2} \\
&\lesssim \|F(v)\|_{L^{b'}([t_1,t_2],L^{b'})} \lesssim \|v\|^{\nu}_{L^p([t_1,t_2],L^p)} \rightarrow 0
\end{align*}
as $t_1,t_2 \rightarrow +\infty$. This implies
\begin{align}
\|[e^{-t_2A} V(t_2)-e^{-t_1A} V(t_1)]\|_{\dot{H}^{\gamma_{\text{w}}}}= \Big\| \Big[\int_{t_1}^{t_2} e^{-sA} G(V(s))ds \Big] \Big\|_{\dot{H}^{\gamma_{\text{w}}}} \rightarrow 0  \nonumber
\end{align}
as $t_1,t_2 \rightarrow + \infty$. The same argument as in Item 1 proves the scattering property for the global solution. The proof of Theorem $\ref{theorem local wellposedness critical scattering small data wave}$ is complete.
\defendproof 
\paragraph{Proof of Theorem $\ref{theorem local wellposedness wave H sigma critical}$.} The proof is similar the one of Theorem $\ref{theorem local wellposedness critical scattering small data wave}$. We thus give a sketch of the proof. We emphasize that here $\nu= 1+ 4\sigma/ (d-2\sigma)$ with $\sigma$ as in $(\ref{H sigma critical wave condition})$. We will solve (NLFW) in
\[
Y:= \Big\{ v \in C(I, \dot{H}^\sigma) \cap C^1(I,L^2) \cap L^\nu(I, L^{2\nu})\ | \ \|[v]\|_{L^\infty(I, \dot{H}^\sigma)} \leq M, \|v\|_{L^\nu(I, L^{2\nu})} \leq N \Big\}
\]
equipped with the distance
\[
d(v,w):= \|[v-w]\|_{L^\infty(I,\dot{H}^\sigma)}  + \|v-w\|_{L^\nu(I, L^{2\nu})},
\]
where $I=[0,T]$ and $M, N>0$ will be chosen later. It is easy to check that under the assumption $(\ref{H sigma critical wave condition})$, $(\nu, 2\nu)$ is admissible with $\gamma_{\nu,2\nu}=\sigma=\gamma_{1,2}+2\sigma$. The Strichartz estimate $(\ref{usual inhomogeneous strichartz wave})$ then implies $\|v_{\text{hom}}\|_{L^\nu(I, L^{2\nu})} \lesssim \|[v](0)\|_{\dot{H}^\sigma}$. Thus $\|v_{\text{hom}}\|_{L^\nu(I, L^{2\nu})}  \leq \varepsilon$ for some $\varepsilon>0$ small enough provided $T$ is small or $\|[v](0)\|_{\dot{H}^\sigma}$ is small. The Strichartz estimate $(\ref{usual inhomogeneous strichartz wave})$ also gives 
\[
\|v_{\text{inh}}\|_{L^\nu(I, L^{2\nu})} \lesssim \|F(v)\|_{L^1(I,L^2)}= \|v\|^\nu_{L^\nu(I,L^{2\nu})}.
\]
Similarly,
\begin{align}
\|F(v)-F(w)\|_{L^1(I,L^2)} \lesssim \Big(\|v\|^{\nu-1}_{L^\nu(I, L^{2\nu})}+ \|w\|^{\nu-1}_{L^\nu(I, L^{2\nu})} \Big) \|v-w\|_{L^\nu(I, L^{2\nu})}. \label{uniqueness critical H sigma wave}
\end{align}
Thus for all $v, w \in Y$, there exists $C>0$ independent of $(\varphi, \phi)\in \dot{H}^\sigma \times L^2$ such that
\begin{align*}
\|\Psi(v)\|_{L^\nu(I, L^{2\nu})}  &\leq \varepsilon +CN^\nu, \\
\|[\Psi(v)]\|_{L^\infty(I, \dot{H}^{\sigma})} &\leq C\|[v](0)\|_{\dot{H}^{\sigma}} +CN^\nu, \\
d(\Psi(v),\Phi(w)) &\leq CN^{\nu-1} d(v,w).
\end{align*}
Now by setting $N=2\varepsilon$ and $M=2C\|[v](0)\|_{\dot{H}^\sigma}$ and choosing $\varepsilon>0$ small enough, we have the existence of solution $v \in Y$ to (NLFW). The uniqueness in $C(I, \dot{H}^{\sigma}) \cap C^1(I,L^2) \cap L^\nu(I,L^{2\nu})$ follows as in Theorem $\ref{theorem local wellposedness subcritical schrodinger sigma <2}$ by using $(\ref{uniqueness critical H sigma wave})$. Here $\|v\|_{L^\nu(I,L^{2\nu})}$ can be small as $T$ is small.\newline
\indent The scattering property is very similar as in the proof of Theorem $\ref{theorem local wellposedness critical scattering small data wave}$. We have
\[
\Big\| \int_{t_1}^{t_2} e^{\pm is\Lambda^\sigma} F(v(s)) ds \Big\|_{L^2} \leq \|F(v)\|_{L^1([t_1,t_2],L^2)} = \|v\|^\nu_{L^\nu([t_1,t_2],L^{2\nu})} \rightarrow 0
\]
as $t_1,t_2 \rightarrow +\infty$. This implies
\begin{align}
\|[e^{-t_2A} V(t_2)-e^{-t_1A} V(t_1)]\|_{\dot{H}^{\sigma}} \rightarrow 0 \nonumber
\end{align}
as $t_1,t_2 \rightarrow + \infty$. This completes the proof.
\defendproof
\section*{Acknowledgments}
The author would like to express his deep gratitude to Prof. Jean-Marc BOUCLET for the kind guidance, encouragement and careful reading of the manuscript. 
\addcontentsline{toc}{section}{Acknowledments}

{\sc Institut de Math\'ematiques de Toulouse, Universit\'e Toulouse III Paul Sabatier, 31062 Toulouse Cedex 9, France.} \\
\indent Email: \href{mailto:dinhvan.duong@math.univ-toulouse.fr}{dinhvan.duong@math.univ-toulouse.fr}
\end{document}